\documentclass[11pt,letterpaper]{amsart}
\usepackage[english]{babel}
\usepackage{amsmath}
\usepackage{amsfonts}
\usepackage{amssymb}
\usepackage{bm}
\usepackage{amsthm}

\usepackage[english]{babel}
\usepackage{yfonts}
\usepackage[T1]{fontenc}
\usepackage[utf8x]{inputenc}
\usepackage{enumerate}
\usepackage{verbatim}
\usepackage{graphicx}
\usepackage{verbatim}
\usepackage{faktor}
\usepackage{xcolor}
\usepackage{xfrac}
\usepackage{tikz}
\usepackage[all]{xy}
\usepackage{mathdots}
\usepackage{cancel}

\newcommand{\calA}{\mathcal A}

\newcommand{\calF}{\mathcal F}

\newcommand{\PSL}{\mathsf{PSL}}
\renewcommand{\r}{\rho}

\renewcommand{\P}{{\rm P}}

\newcommand{\ov}{\overline}
\newcommand{\Stab}{\rm Stab}

\newcommand{\OO}{\mathsf O}

\newcommand{\di}{\partial_{\infty}}
\newcommand{\Gr}{{\rm Gr}}

\newcommand{\SO}{{\rm SO}}

\newcommand{\g}{\gamma}

\newcommand{\G}{\Gamma}
\newcommand{\K}{\mathbb K}
\newcommand{\T}{\Theta}

\newcommand{\Id}{{\rm Id}}
\newcommand{\R}{\mathbb R}
\newcommand{\Z}{\mathbb Z}
\renewcommand{\P}{\mathbb P}

\renewcommand{\H}{\mathbb H}

\newcommand{\N}{\mathbb N}

\newcommand{\SL}{\mathsf{SL}}
\newcommand{\PGL}{\mathsf{PGL}}
\newcommand{\PO}{\mathsf{PO}}

\newcommand{\Sp}{{\rm Sp}}

\newcommand{\<}{\langle}
\renewcommand{\>}{\rangle}

\newcommand{\Hom}{{\rm Hom}}
\newcommand{\Homr}{{\rm Hom^{red}}}

\newcommand{\im}{{\rm Im}}

\newcommand{\sfw}{\mathsf{w}}
\newcommand{\sfz}{\mathsf{z}}
\newcommand{\sfv}{\mathsf{v}}
\newcommand{\sZ}{\mathsf{Z}}
\newcommand{\sX}{\mathsf{X}}
\newcommand{\sfs}{\mathsf{s}}
\newcommand{\sfy}{\mathsf{y}}

\newcommand{\dg}{\partial_{\infty}\Gamma}

\newcommand{\cro}{{\rm cr}}
\newcommand{\gcr}{{\rm cr}}

\newcommand{\tv}{\hspace{1mm}\pitchfork\hspace{1mm}}
\newcommand{\ntv}{\hspace{1mm}\cancel{\pitchfork}\hspace{1mm}}

\newcommand{\Is}{{\rm Iso}}
\newcommand{\spa}{{\rm span}}

   \newcommand\quotient[2]{
	\mathchoice
	{
		\text{\raise.5ex\hbox{$#1$}\big/\lower.5ex\hbox{$#2$}}%
	}
	{
		\text{\raise.25ex\hbox{$#1$}\big/\lower.25ex\hbox{$#2$}}%
	}
	{
		#1\,/\,#2
	}
	{
		#1\,/\,#2
	}
}

\newcommand{\Isom}{\text{Isom}}

\newcommand{\bpm}{\begin{pmatrix}}
\newcommand{\epm}{\end{pmatrix}}

\theoremstyle{plain}
\newtheorem{thm}{Theorem}[section]
\newtheorem{lem}[thm]{Lemma}
\newtheorem{prop}[thm]{Proposition}
\newtheorem{cor}[thm]{Corollary}

\newtheorem*{thm*}{Theorem}

\newtheorem*{notation}{Notation}
\newtheorem{thmA}{Theorem}

\theoremstyle{definition}
\newtheorem{example}[thm]{Example}

\newtheorem{defn}[thm]{Definition}
\newtheorem{remark}[thm]{Remark}

\newcommand{\thismonth}{\ifcase\month 
  \or January\or February\or March\or April\or May\or June%
  \or July\or August\or September\or October\or November%
  \or December\fi}

\title{Positive surface group representations in $\PO(p,q)$}
\author{Jonas Beyrer and Beatrice Pozzetti}
\date{\today}
\begin{document}
\thanks{J.B. acknowledges funding by the Deutsche Forschungsgemeinschaft (DFG, German Research Foundation), 338644254 (SPP2026), and the Schweizerischer Nationalfonds (SNF, Swiss Research Foundation), P2ZHP2 184022 (Early Postdoc.Mobility). B.P. acknowledges funding by the DFG, 427903332 (Emmy Noether), and is partially supported by the DFG under Germany's Excellence Strategy EXC-2181/1-390900948. Both authors acknowledge funding by the DFG, 281869850 (RTG 2229). J.B. thanks I.H.E.S. for their hospitality.}
\begin{abstract}
 We show that $\Theta$-positive Anosov representations $\rho:\G\to\PO(p,q)$ of a surface group $\G$ satisfy root versus weight collar lemmas for all the Anosov roots, and are positively ratioed with respect to all such roots. We deduce from this, using a result of  \cite{Beyrer-Pozzetti2}, that $\Theta$-positive Anosov representations $\rho:\G\to\PO(p,q)$ form connected components of character varieties.
\end{abstract}

\maketitle
\tableofcontents
\section{Introduction}

\addtocontents{toc}{\protect\setcounter{tocdepth}{1}}
Higher rank Teichm\"uller theory stems from the seminal work of Labourie \cite{Labourie-IM},  Fock--Goncharov \cite{FG}, and Burger--Iozzi--Wienhard \cite{BIW}; they discovered that, for some classes of Lie groups, there exist \emph{higher rank Teichm\"uller spaces:} connected components of character varieties of fundamental groups $\G$ of closed surfaces $S$ of genus at least 2 in higher rank semisimple Lie groups that only consist of injective representations with discrete image. More specifically Labourie and Fock--Goncharov showed that for split real Lie groups the \emph{Hitchin components}, discovered by Hitchin \cite{Hitchin}, form higher rank Teichmüller spaces, while Burger--Iozzi--Wienhard discovered the \emph{maximal components} for Hermitian Lie groups and proved that they also form higher rank Teichmüller spaces. These components consist only of representations with further remarkable geometric properties bearing strong similarities with holonomies of hyperbolizations \cite{Lee-Zhang, BP, Lab-McShane, Vlamis-Yarmola, HuangSun, FP}; furthermore, since they form connected components of character varieties, they can be studied with an array of different tools including Higgs bundles \cite{Hitchin, HiggsBundles} and real algebraic geometry \cite{BIPP}. 

A recent breakthrough in the field was given by the insight of Guichard--Wienhard \cite{GWpositivity}. Partially in collaboration with  Labourie, they developed a beautiful and clear conjectural picture of all higher rank Teichm\"uller spaces \cite{GLW, GWPaper}. They give a complete list of pairs of Lie groups and parabolic subgroups that admit a \emph{$\T$-positive structure}\footnote{In Guichard--Wienhard's theory the letter $\Theta$ refers to the subset of simple roots associated to the parabolic subgroup playing a role for the notion of positivity, 
	but for the purpose of the paper it is harmless to understand "$\Theta$-positive" as "positive in the sense of Guichard--Wienhard".}; when this is the case they define  \emph{$\T$-positive representations}. They conjecture that higher rank Teichm\"uller spaces are precisely the connected components of character varieties that contain a $\Theta$-positive representation \cite[Conjecture 5.4]{GWpositivity}, and understand important positivity features of the limit set of $\Theta$-positive representations that, generalizing \cite{FG}, are behind the good geometric properties of these representations.
{ Simultaneously Bradlow--Collier--Garcia-Prada--Gothen--Oliveira, partially in collaboration with  Aparicio-Arroyo, discovered and parametrized special connected components of the moduli space of Higgs bundles on a compact Riemann surface $X$ that conjecturally detected all examples of higher rank Teichm\"uller spaces, up to some exceptional Lie groups \cite{ABC-IM,BCGPGO}. 
}

In this work we study aspects of the higher rank Teichm\"uller theory of $\PO(p,q)$: the classification of \cite{GWPaper} gives that $\PO(p,q)$  is the only family of  classical groups that carries a positive structure besides split real Lie groups and  Hermitian Lie groups of tube type; in the last two cases $\Theta$-positive representations are, respectively, Hitchin and maximal representations. 
We make three major contributions: 
we show that 
every $\T$-positive  Anosov representation satisfies collar lemmas comparing roots and weights,  we show that every $\T$-positive Anosov representation is positively ratioed, and, also relying on results of the companion paper \cite{Beyrer-Pozzetti2}, we confirm Guichard--Wienhard's conjecture  that $\T$-positive Anosov representations form connected components of character varieties.\footnote{In an independent paper  Guichard--Labourie--Wienhard proved that $\T$-positive representations for any admissible target $\sf G$ form connected components of the set of non-parabolic representations \cite{GLW}. Relying on the aforementioned results by Aparicio-Arroyo--Bradlow--Collier--Garcia-Prada--Gothen--Oliveira \cite{ABC-IM, BCGPGO} this gives examples of higher rank Teichm\"uller spaces.} 

The first two results generalize familiar properties of holonomies of hyperbolizations to $\T$-positive representations, and can be regarded as additional geometric justification for the denomination \emph{higher rank Teichm\"uller theory}. The third result gives the first proof that a connected component of a character variety in a higher rank Lie group consists only of discrete and injective representations not relying on Higgs bundles\footnote{Labourie's proof that Hitchin representations are injective relies on Higgs bundles techniques to guarantee that Hitchin representations are irreducible and thus non-parabolic \cite[Lemma 10.1]{Labourie-IM}.} or bounded cohomology. The new strategy we develop could also be applied to Hitchin representations in $\PSL(n,\R)$, after extracting a suitable formulation of a root versus weight collar lemma  from \cite[Proposition 2.12]{Lee-Zhang}. All our results also hold and are new for Hitchin representations into $\PO(p,p)$, see Appendix~\ref{a.POpp} for details.

\subsection*{$\T$-Positive representations are positively ratioed}
To describe our results in more detail, we need to recall some basic facts from Guichard--Wienhard's theory of $\T$-positivity in the special case of the group $\PO(p,q)$. For technical reasons, we assume from now on that $1< p< q$. However, the case $p=q$ works similarly, we discuss it in Appendix~\ref{a.POpp}. We denote by $\calF_{p-1}(\R^{p,q})$ the set of  partial flags of isotropic subspaces of $\R^{p,q}$, which consist of isotropic subspaces of all dimensions but the maximal one (see Section~\ref{ss.pos}). 
The $\T$-positive structure associated to $\PO(p,q)$ is defined with respect to the stabilizer of a point in $\calF_{p-1}(\R^{p,q})$. Since $\PO(p,q)$ has a  $\T$-positive structure,  one can define \emph{$\T$-positive $n$--tuples} in $\calF_{p-1}(\R^{p,q})$. This generalizes the notion of cyclically oriented $n$--tuple on $\R\P^1\simeq \di \H^2$, which is crucial in classical Teichm\"uller theory. A map $\xi:\mathbb{S}^1\to \calF_{p-1}(\R^{p,q})$ is \emph{$\T$-positive} if it maps positive $n$--tuples in $\mathbb{S}^1$ to positive $n$--tuples in $\calF_{p-1}(\R^{p,q})$. {{Let $\G$ be the fundamental group of a closed hyperbolic surface, so that in particular $\dg\simeq\mathbb{S}^1$. }}
Following \cite[Definition 5.3]{GWpositivity} we say that a representation $\rho:\G\to\PO(p,q)$ is \emph{$\Theta$-positive} if it admits an equivariant $\T$-positive map.
In this paper we will only consider $\T$-positive representations that are furthermore Anosov\footnote{In the aforementioned independent paper  Guichard--Labourie--Wienhard proved that a $\T$-positive representation is necessarily Anosov, so this assumption is not really needed \cite{GLW}. However since our two papers are independent we will keep the assumption.}; while any $\T$-positive map is automatically transverse, in order to be Anosov it needs to additionally be continuous and dynamics preserving (see Section~\ref{s.Anosov} for the precise definition).

The first result of the paper establishes an additional positivity property of $\T$-positive representations: they are \emph{positively ratioed} in the sense of Martone--Zhang (\cite[Definition 2.25]{MZ}, see also Definition~\ref{d.posrat}). This amounts to saying that the restriction of the natural cross ratio on the set of $k$-dimensional isotropic planes $\Is_k(\R^{p,q})$ to the image of the boundary map induces a positive cross ratio; a $\Theta$-positive boundary map induces boundary maps to the isotropic $k$-planes for $k<p$ through the natural projection $\calF_{p-1}(\R^{p,q})\to \Is_k(\R^{p,q})$. 

\begin{thmA}\label{thm.intro-positive}(Theorem~\ref{thm.pos ratioed})
	Let $\r:\G\to\PO(p,q)$ be a $\T$-positive Anosov representation, then $\rho$ is $k$-positively ratioed for all $k< p$.
\end{thmA}
Being positively ratioed implies that suitable Finsler length functions associated to the representations can be computed as intersections with a geodesic current \cite[Theorem 1.1]{MZ}. In turn this has a number of geometric consequences:  length shortening under surgery \cite[Corollary 1.3]{MZ}, relations between systole and entropy \cite[Corollary 1.2]{MZ}, as well as domains of discontinuity for the mapping class group action on suitable compactifications \cite[Corollary 1.6]{BIPP0}.

It follows from our results of  \cite{Beyrer-Pozzetti} that $\T$-positive representations into $\PO(p,q)$ are $k$--positively ratioed for $k\leq  p-2$; thus Theorem~\ref{thm.intro-positive} is only new for $k=p-1$. However in this case a new substantial difficulty needs to be overcome, as the boundary maps are only Lipschitz regular and not $C^1$. As a result, rather than relying on continuity of the derivative, we need to perform a much more careful analysis. 
{{ We also include a new proof for $k\leq  p-2$, since it is better adapted to the $\T$-positive structure and builds on two results of independent interest: First, for all $k<p$, we show that any $\T$-positive triple $(x,y,z)$ in $\calF_{p-1}(\R^{p,q})$ naturally defines a tangent cone $c^+_k(x)$ in ${\rm T}_x \Is_k(\R^{p,q})$, and the induced boundary map of a $\T$-positive representation containing $(x,y,z)$ is almost everywhere tangent to $c^+_k(x)$ (Proposition~\ref{p.p-1derivative}, see also Remark~\ref {r.derTheta}). This relies on the fact that the boundary map is Lipschitz regular as shown in \cite{PSW2}, and was inspired from discussions of the second author with Wienhard while working on \cite{PSW2}. Second, we prove that the cross ratio (infinitesimally) increases along those cones (Proposition~\ref{prop.derivative cross ratio for positive triple: general}).}}

\subsection*{Collar lemmas}
An important property of holonomies of hyperbolizations is the collar lemma \cite{Keen}: any simple curve $g$ admits an embedded collar neighbourhood of width logarithmic in the inverse of the hyperbolic length of $g$. In particular this has the algebraic consequence that  the length of any curve $h$ crossing $g$  must be at least the logarithm of the inverse of the length of  $g$. As a result, only simple curves can be very short in a hyperbolic structure. We generalize here the algebraic formulation of the collar lemma to $\T$-positive representations and prove an asymmetric strengthening: we show that the $k$-th Finsler length, which is the logarithm of the product of the first $k$ eigenvalues, of an element $g$ already controls the $k$-th eigenvalue gap of all linked elements $h$. 
We say that two elements $g,h\in\G$ are \emph{linked} if the attracting and repelling fixed points $h_{+},h_{-}\in \dg$ of $h$ are in different connected components of $\dg\backslash\{g_-,g_+\}$. Given $A\in\PO(p,q)$ we denote by $\lambda_1(A),\ldots,\lambda_{p+q}(A)$ the generalized eigenvalues of a lift of $A$ to ${\rm O}(p,q)$ ordered so that their moduli are non-decreasing, i.e. $|\lambda_i(A)|\geq |\lambda_{i+1}(A)|$.
\begin{thmA}\label{thm.intro-collar}(Theorem~\ref{t.collar})
	Let $\r:\G\to\PO(p,q)$ be a $\T$-positive Anosov representation and $g,h\in\G$ a linked pair. Then for any $k\leq p-1$
	\begin{align*}
		\left(1-\left|\frac{\lambda_{k+1}}{\lambda_{k}}(\r(h))\right|\right)^{-1}<\lambda_1^2 \cdots\lambda_k^2 (\r(g)).
	\end{align*}
\end{thmA}
{{
We can rephrase the statement in Lie theoretic terms: Let $\alpha_k(\r(h))$ be the $k$-th  \emph{restricted root} of $\PO(p,q)$ (for the standard numeration) applied to the Jordan projection of $\r(h)$, so that $\alpha_k(\r(h))=\log\left|\frac{\lambda_{k}}{\lambda_{k+1}}(\r(h))\right|$, and let $\omega_k(\r(g))$ be the $k$-th \emph{fundamental weight}  applied to the Jordan projection of $\r(g)$, namely $\omega_k(\r(g))=\log(\lambda_1^2 \cdots\lambda_k^2 (\r(g)))$.  Theorem~\ref{thm.intro-collar}, after elementary algebraic operations, states that  for any linked pair $g,h$ and $k\leq p-1$,
\begin{equation}\label{e.col2}
	\left(e^{\alpha_k(\rho(h))}-1\right)\left(e^{\omega_k(\rho(g))}-1\right)>1.
\end{equation}
{In higher rank many measures of the magnitude of an element play an important role in understanding geometric features of actions, and different measures generalize different properties of hyperbolizations. On the one hand the fundamental weights $\omega_k$ describe the translation length on the symmetric space with respect to suitable Finsler distances \cite{KLP}, and, as mentioned above, behave like the hyperbolic length function under surgery for representations in higher rank Teichm\"uller spaces. On the other hand the roots $\alpha_k$, albeit not being induced by a distance, are, for Anosov representations, coarsely equivalent to the stable length with respect to any generating system \cite{PK}, and their entropy is constant and equal to one on higher rank Teichm\"uller spaces \cite{PS,PSW2}.  Adding to this, {{Theorem~\ref{thm.intro-collar} encodes a powerful generalization of another feature of holonomies of hyperbolizations -- the collar lemma -- to $\T$-positive representations into $\PO(p,q)$; analogous results were previously established for  Hitchin representations \cite{Lee-Zhang}, maximal representations \cite{BP} and representations that satisfy some partial hyperconvexity properties, a class containing representations that are not $\T$-positive, and does not always consist of connected components of the character variety \cite{Beyrer-Pozzetti}}.}

{{
The asymmetry comparing two different length functions in Theorem~\ref{thm.intro-collar} is key and indicates that both  $\alpha_k$ and $\omega_k$ play an important role in the study of $\T$-positive representations. Equation \eqref{e.col2}  implies that also a weight versus weight collar lemma 
holds:
for any linked pair $g,h$ and $k\leq p-1$ we have
\begin{equation*}
	\left(e^{\omega_k(\rho(h))}-1\right)\left(e^{\omega_k(\rho(g))}-1\right)>1.
\end{equation*}
This generalizes the collar lemma from \cite{CTT} in the case of $\PO(2,q)$.
Since it is not expected that a root versus root collar lemma holds, cfr. \cite[Section 7]{Beyrer-Pozzetti},  the asymmetric version is the strongest version to hope for.  We will indeed need this strong version, as in the proof of Theorem~\ref{thm.C-intro} it is crucial that the collar lemma allows us to bound the root length from below.

Theorem~\ref{thm.intro-collar} follows our the results in \cite{Beyrer-Pozzetti} for $k<p-2$, and is  new for $k=p-2,\,p-1$. 
{{To deal with the two additional cases we need a different approach compared to \cite{Beyrer-Pozzetti} since $\Theta$-positive representations are in general not $p$--Anosov and don't satisfy additional transversality properties. There are three key new tools, building on the $\T$-positive structure, that allow us to circumvent this problem. First we need that the cross ratio is increasing along cones defined by $\T$-positive triples, as mentioned above. 
The second tool is the construction of a \emph{hybrid} flag associated to a pair of transverse flags in $\calF_{p-1}$  (Definition~\ref{def.hybrid flag}). In Proposition~\ref{prop.k-projection is positive} we establish in which sense positivity is preserved under the hybrid construction. This result is of independent interest, and we believe it will be useful in further study of geometric properties of $\T$-positive representations and more general classes of Anosov representations. 
With Proposition~\ref{prop.k-projection is positive} at hand we can reduce the proof of Theorem~\ref{thm.intro-collar} to an inequality for the $\T$-positive structure for $\PO(2,q)$, which we establish in Lemma~\ref{p.collar2,n}, and which is the third key building block of our proof. This third step (i.e. Lemma~\ref{p.collar2,n}) constitutes the main step in the proof of a root versus weight collar lemma for maximal representations in $\PO(2,q)$ (Theorem~\ref{t.collarmax}), which is substantially different to the available proof for $\Sp(2n,\R)$ \cite{BP}.
		
		As it does not cost any extra effort we include proofs of the collar lemmas for the cases $k<p-2$ as well, which are simpler and more direct than the ones in \cite{Beyrer-Pozzetti}, making this work  independent from \cite{Beyrer-Pozzetti}.}}

\subsection*{$\T$-Positive Anosov representations form connected components}
A key advantage of a collar lemma controlling eigenvalue gaps, is that it guarantees that a limit of representations satisfying such collar lemma remains proximal. This is particularly important since we showed in the companion paper \cite{Beyrer-Pozzetti2} that a proximal limit of positively ratioed representations admits  continuous, dynamics preserving equivariant boundary maps \cite[Theorem B]{Beyrer-Pozzetti2}. Since thanks to Theorems~\ref{thm.intro-positive} and~\ref{thm.intro-collar} we can apply  such result,  we obtain:
\begin{thmA}\label{thm.C-intro}
	The set of $\T$-positive Anosov representations is closed in $\Hom(\G, \PO(p,q))$.
\end{thmA} 
 We show that a limit of $\T$-positive Anosov representations is $\Theta$-positive Anosov in two steps. First we use Theorems~\ref{thm.intro-positive} and~\ref{thm.intro-collar} to show that  \cite[Theorem B]{Beyrer-Pozzetti2} is applicable. This guarantees that the limit representation admits continuous, dynamics preserving equivariant boundary maps, which we prove being additionally $\Theta$-positive using properties of $\Theta$-positivity. In the second step of the proof we show that the representation is Anosov, by using a criterion due to Gueritaud--Guichard--Kassel--Wienhard \cite{GGKW} based on eigenvalue gap growth. Here we use once again the positivity of hybrid flags, and we introduce a new idea, allowing us to read the eigenvalue gap of an element as a cross ratio involving an hybrid flag; we  use this to control the growth of the eigenvalue gaps 
	(see Proposition~\ref{p.positiveAnosov} for details).

{{Since by the work of Guichard--Wienhard positivity of $n$-tuples is an open condition \cite[Theorem 4.7]{GWpositivity} (see also Corollary~\ref{lem.deform positive tuples} below),}} and Anosov representations are open \cite[Theorem 1.2]{Guichard-Wienhard-IM}, the set of $\T$-positive Anosov representations is open in the character variety. A consequence of Theorem~\ref{thm.C-intro} is thus the following.
\begin{cor}\label{cor.intro}(Corollary~\ref{c.higherTM spaces})
	Being $\T$-positive Anosov is an open and closed condition in the character variety
	$$\Xi (\G,\PO(p,q)):=\quotient {\Homr(\G,\PO(p,q))}{\PO(p,q)}$$
	and thus constitutes connected components.
\end{cor}
{Recall that the character variety $\Xi (\G,\PO(p,q))$ can be defined equivalently as the semialgebraic GIT quotient of the representation space $\Hom(\G,\PO(p,q))$ by the $\PO(p,q)$ action by conjugation \cite{RS, BIPPrsp}, or  by the quotient of the subset $\Homr(\G,\PO(p,q))\subset \Hom(\G,\PO(p,q))$ consisting of reductive representations, on which the $\PO(p,q)$-action is separated \cite[S. 4.1]{APcomp}.}
\begin{remark}\label{rem.pos general groups}
A representation into $\OO(p,q)$\footnote{Observe that, when $p+q$ is odd, $\SO(p, q) = \PO(p, q)$ and $\OO(p, q) = \SO(p, q) \times \Z/2\Z$, but that things are more subtle when $p+q$ is even, and delicate liftability questions arise.}, a group that also naturally acts on $\calF_{p-1}(\R^{p,q})$, is $\T$-positive if and only if its projectivization is positive. In particular all the results in this work hold also for $\T$-positive representations in $\OO(p,q)$. The case where the image is in $\SO(p,q)$ is particularly interesting, since for those groups connected components of character varieties have been studied in detail with Higgs bundles techniques \cite{ABC-IM} (see Remark~\ref{rem.ht-components}).
\end{remark}

{Combining Corollary~\ref{cor.intro}  with the work  \cite[Theorem 7.6 and Proposition 7.13]{ABC-IM} we obtain that a conjugacy class of reductive representations $\rho:\G\to\SO(p,q)$ consists of $\Theta$-positive Anosov representations if and only if it belongs to one of the special connected components parametrized in \cite[Theorem 4.1]{ABC-IM}. This settles
	\begin{thm*}[cfr. {\cite[Section 7.2]{ABC-IM}}]
		Let $X$ be the choice of a Riemann surface structure on $S$. The subset of $\Xi (\G,\SO(p,q))$ consisting of conjugacy classes of reductive $\Theta$-positive Anosov representations is parametrized by 
		$$\mathcal M_{K^p}(\SO(1,q-p+1))\times\bigoplus_{j=1}^{p-1}H^0(K^{2j})$$
		where $K$ is the canonical bundle of $X$, $\mathcal M_{K^p}(SO(1,q-p+1))$ denotes the moduli space of $K^p$-twisted $\SO(1,q-p+1)$-Higgs bundles on $X$ and $ H^0(K^{2j})$ is the vector space of holomorphic sections of $K^{2j}$.
	\end{thm*}	
}
{
\subsection*{$\T$-positive representations of open surfaces}
It is possible to define good notions of $\T$-positive representations of more general surfaces that are not necessarily compact, by requiring the existence of a positive boundary map defined on some associated cyclically ordered set such as the circle with the action induced by a finite volume hyperbolization, or the subset of cusps (cfr. \cite{FG} for representations in split real Lie groups and \cite{BIW} for maximal representations in Hermitian Lie groups). We expect that Theorem~\ref{thm.intro-collar} works verbatim in this setting, and that Theorem~\ref{thm.intro-positive} and Theorem~\ref{thm.C-intro} admit suitable generalizations, respectively that the natural pullback of the cross ratio is positive, and that the space of $\T$-positive representations forms connected components of relative character varieties. 

Some of our techniques require however that the surface is closed: we use differential arguments based on the regularity of the image of the boundary map, which cannot be directly generalized to the open surfaces where, in most cases, the boundary map is not even continuous. While new ideas are needed to treat the general case, we expect that our general strategy, as well as some of the tools  we develop, such as the study of hybrid flags, and the discovery of cross ratios that read eigenvalue gaps, will also be precious for dealing with open surfaces.}

\subsection*{Relation with Guichard--Wienhard's work}
{ We build on Guichard--Wienhard's theory of positivity for general Lie groups $\sf G$, which generalizes aspects of Lusztig positivity \cite{Lus} to this setting. We freely use results from \cite{GWPaper} where general properties of positive $n$-tuples of flags, as well as of the positive semigroup are established.  In a previous version of this work we re-established many properties of $\T$-positive triples only building on results announced in \cite{GWpositivity}, but now that \cite{GWPaper} is available we decided to avoid duplicates.}
As already remarked, we also build upon \cite[Theorem D]{PSW2},  by the second author in collaboration with Sambarino and Wienhard, and the ideas behind positivity that are crucial in that proof.

{{Our work is  independent from the work of Guichard--Labourie--Wienhard, who proved that there exist higher rank Teichm\"uller components consisting of $\T$-positive representations in any Lie group $\sf G$ admitting a $\T$-positive structure \cite{GLW}. In their paper they  prove that general $\Theta$-positive representations are necessarily $\Theta$-Anosov, and  form connected components of the subset of representations that are not contained in a parabolic subgroup.
Since we show that $\Theta$-positive Anosov representations are closed in the whole representation variety, our papers are complementary and prove together \cite[Conjecture 5.4]{GWpositivity} for the group $\PO(p,q)$. Combining our results with \cite{ABC-IM} we obtain a stronger statement than \cite[Conjecture 5.4]{GWpositivity}: for $1< p< q$ the only higher rank Teichm\"uller components of the $\PO(p,q)$-character variety are the ones which contain a $\T$-positive representations (see also Remark~\ref{rem.ht-components}).}}

}
 \subsection*{Acknowledgements}
Beatrice Pozzetti is indebted to Anna Wienhard for sharing her ideas on $\Theta$-positive representations, and numerous discussions that helped her shaping and sharpening her understanding of these representations. We thank the anonymous referees for their insightful remarks which helped improving the exposition.


 \section{Preliminaries}\label{s.prelims}
\addtocontents{toc}{\protect\setcounter{tocdepth}{2}}
We list here some notation that we keep throughout this work\\

\noindent{\it The group $\G$}
\begin{itemize}
	\item  $\G$ denotes a \emph{surface group}, i.e. the fundamental group of a closed connected orientable surface of genus at least 2. Its \emph{Gromov boundary} $\dg$ is homeomorphic to the circle $\mathbb{S}^1$.
	
	\item $\dg^{(j)}$ denotes the set of $j$--tuples of $\dg$ consisting of pairwise distinct points.
	\item$\dg^{[j]}\subset\dg^{(j)}$ denotes the set of cyclically ordered $j$--tuples, namely the $j$--tulples that are either positively or negatively ordered for the standard cyclic orientation of the circle.
	
	\item $(x,y)_z\subset \dg$ is the interval of $\dg\backslash\{x,y\}$ that does \emph{not} contain $z$, for $(x,y,z)\in \dg^{(3)}$ \cite{MZ}.
\end{itemize}

\noindent{\it Isotropic subspaces and eigenvalues}

\begin{itemize}
	\item $\R^{p,q}$ is the vector space $\R^{p+q}$ equipped with a non-degenerate symmetric bilinear form $Q$ of signature $(p,q)$ -- except from the appendix, we will assume $1<p< q$. We denote by $\cdot^{\perp}$ the orthogonal complement with respect to $Q$. The isometry group of $\R^{p,q}$ is denoted by $\OO(p,q)$.

	\item{{ $\lambda_1(g),\ldots,\lambda_{p+q}(g)$ denote the  eigenvalues of an element $g\in\OO(p,q)$ counted with multiplicity and ordered so that their absolute values are  non-increasing, i.e. $|\lambda_i(g)|\geq |\lambda_{i+1}(g)|$. Since $g\in \OO(p,q)$, it follows that $\lambda_i(g)=\lambda^{-1}_{p+q+1-i}(g)$. If $g\in\PO(p,q)$ denote by $\tilde{g}\in\OO(p,q)$ a lift, then $$\frac{\lambda_i}{\lambda_{i+1}}(g):=\frac{\lambda_i(\tilde{g})}{\lambda_{i+1}(\tilde{g})}$$
			does not depend on the choice of the lift.}}

	\item  $\Is_k(\R^{p,q})$ denotes the set of  \emph{isotropic $k$-planes} in $\R^{p,q}$, i.e. $k$--planes on which the form $Q$ is identically zero. Given $V\in \Is_k(\R^{p,q})$ and $W\in \Is_k(\R^{p,q})$ we say that $V$ and $W$ are \emph{transverse} if the sum $V+ W^{\perp}$ is direct (equivalently $V^{\perp}+W$ is direct); we denote this by $V\tv W$. If two subspaces are not transverse we write $V\ntv W$.

\end{itemize}

\subsection{Anosov representations}\label{s.Anosov}
The \emph{stable length} of $\g\in\G$ is $|\g|_{\infty}:=\lim_{n\to\infty} |\g^n|_{\G}\slash n$, where $|\cdot|_{\G}$ is a fixed word metric on the Cayley graph of $\G$. 
Anosov representations in $\PO(p,q)$ admit the following characterization \cite[Theorem 1.7]{GGKW} which we will use as a definition.

\begin{defn}\label{d.Anosov}
	A homomorphism $\r:\G\to \PO(p,q)$ is  \emph{$k$--Anosov}  for $k\in\{1,\ldots,p\}$, if there exists a $\r$-equivariant continuous boundary map $\xi^k:\dg\to \Is_k(\R^{p,q})$ such that
	\begin{enumerate}
		\item $\xi^k(x)\tv \xi^{k}(y)$ for all $x\neq y\in \dg$
		\item $\xi^k$ is \emph{dynamics preserving}, i.e. for every infinite order element $\g\in\G$ with attracting fixed point $\g_{+}\in\dg$ we have that $\xi^k(\g_{+})$ is an attracting fixed point for the action of $\r(\g)$ on $\Is_k(\R^{p,q})$.
		\item $\left|\frac{\lambda_k}{\lambda_{k+1}}(\r(\g_i))\right|\to \infty$ if $|\g_i|_{\infty}\to\infty$.
	\end{enumerate}
\end{defn}

Anosov representations have many interesting geometric and dynamical properties, for instance they are discrete and faithful and form an open subset of  $\Hom(\G,\PO(p,q))$. 

\begin{remark}
	Anosov representations are defined for general reductive Lie groups;  we do not introduce the general theory here as we only work with $\PO(p,q)$. After the natural inclusion $\PO(p,q)\to \PGL(\R^{p,q})$, a representation $\r:\G\to \PO(p,q)$ is $k$--Anosov if and only if it is $k$--Anosov in $\PGL(\R^{p,q})$. In particular any result for $k$--Anosov representations in $\PGL(\R^{p,q})$ applies to our context.
\end{remark}
We say that a representation is  \emph{$\T$-Anosov} if it is $k$--Anosov for all $k=1,\ldots,p-1$. Since the boundary maps of Anosov representations are dynamics preserving,  for all non-trivial $\g\in \G$ it holds  $\xi^k(\g_+)\subset\xi^{k+1}(\g_+)$. Thus for any $\T$-Anosov representation the continuity of the boundary map and the density of the fixed points in $\dg$ imply that the map 
$$\xi=(\xi^1,\ldots,\xi^{p-1}):\dg\to\calF_{p-1}$$
  is well defined, equivariant, continuous and transverse. Here, and in the rest of the paper, $\calF_{p-1}$ denotes the partial flag manifold consisting of flags of isotropic subspaces of dimensions $1,\ldots, p-1$.
\begin{notation}
	We will write $x_{\r}^k$ and  $x^k$  for $\xi^k(x)$, where $\r$ is a $k$--Anosov representation and $\xi^k$ the associated boundary map.
	Similarly we may write $\g_{\r}$ instead of $\rho(\g)$ for  $\g\in \G$.
	Moreover we will write $x^{p+q-k}$ for $\xi^k(x)^{\perp}$.
\end{notation}

\subsection{Cross ratios}\label{s.cr}
In this paper we use the cross ratios defined on
\begin{align*}
	\calA_k:=\{(V_1,W_1,W_2,V_2) \in \Is_k^4(\R^{p,q})\:|\: V_i\tv W_j, i,j=1,2 \}.
\end{align*}

\begin{defn}\label{d.cr}
	The cross ratio $\cro_k:\calA_k\to \R\backslash\{0\}$ is defined by 
	\begin{align*}
		\cro_k (V_1,W_1,W_2,V_2):=\frac{V_1\wedge W^{\perp}_2}{V_1\wedge W^{\perp}_1} \frac{V_2\wedge W^{\perp}_1}{V_2\wedge W^{\perp}_2}.
	\end{align*}
	Here $V_i\,\wedge W^{\perp}_j$ denotes the element $v_1\wedge \ldots\wedge v_k \wedge w_1\wedge \ldots \wedge w_{p+q-k}\in \wedge^{p+q} \R^{p+q} \simeq \R$ for bases $\{v_1,\ldots, v_k\},\{w_1,\ldots, w_{p+q-k}\}$  of $V_i$ and $W^{\perp}_j$, and a fixed identification $\wedge^{p+q} \R^{p+q} \simeq \R$. 
	The value of $\cro_k$ is independent of all choices made.
\end{defn}

If $(V_1,W_1,W_2,V_2)\in\calA_1\subset \Is^4_1(\R^{p,q})$, the cross ratio $\cro_1$ can be expressed as
$$\cro_1 (V_1,W_1,W_2,V_1):=\frac{Q(\tilde{V}_1, \tilde{W}_2)}{Q(\tilde{V}_1, \tilde{W}_1)} \frac{Q(\tilde{V}_2, \tilde{W}_1)}{Q(\tilde{V}_2, \tilde{W}_2)},$$
where $\tilde{V}_i,\tilde{W}_j \in \R^{p,q}\backslash\{0\}$ are  such that $\tilde{V}_i\in V_i$ and $\tilde{W}_i\in W_i$.

The following properties are classical and easy to verify.
\begin{prop}\label{prop.property of grassmannian cro}
	Let $V_1,V_2,V_3\in\Is_k(\R^{p,q})$ and $W_1,W_2,W_3\in\Is_k(\R^{p,q})$. Then whenever all quantities are defined we have
	\begin{enumerate}
		\item $\cro_k(V_1,W_1,W_2,V_2)^{-1} =\cro_k (V_2,W_1,W_2,V_1)=\cro_k (V_1,W_2,W_1,V_2)$, in particular $\cro_k(V_1,W_1,W_1,V_2)=1$
		\item $\cro_k (V_1,W_1,W_2,V_2)\cdot\cro_k (V_2,W_1,W_2,V_3)=\cro_k (V_1,W_1,W_2,V_3)$
		\item $\cro_k (V_1,W_1,W_2,V_2)\cdot\cro_k (V_1,W_2,W_3,V_2)=\cro_k (V_1,W_1,W_3,V_2)$

		\item $\cro_k (V_1,W_1,W_2,V_2)=\cro_k (g V_1,g W_1,g W_2,g V_2)\quad  \forall g\in\PO(p,q)$.
		\item The cross ratio is algebraic.
	\end{enumerate}
	The identities $(3)$ and $(4)$ will be called \emph{cocycle identities}. 
\end{prop}
\begin{proof}
 {(1), (2), (3) are straightforward computations from the expression in Definition~\ref{d.cr}, while (4) follows since $\PO(p,q)$ preserves orthogonal complement and induces the multiplication by a scalar on $\wedge^{p+q} \R^{p+q}$. 
 	
(5)  The expression for the cross ratio is clearly algebraic when defined on the frame manifold, and descends to an algebraic function on $\mathcal A_k$ since it doesn't depend on the choice of a lift.}
\end{proof}	

In the context of Anosov representations, the cross ratio computes the fundamental weights. It is easy to check:
\begin{lem}[{\cite[Page 19]{MZ}}]\label{lem.cro-period}
	If $\r$ is $k$--Anosov, then
	$$\cro_k (\g_-^k,x^k,\g x^k,\g_+^{k})=\lambda_1^2(\g)\ldots\lambda_k^2(\g)$$
	for all non-trivial $\g\in \G$ and $x\in \dg\backslash\{\g_{\pm}\}$. 
\end{lem}

A computation as in the proof of  \cite[Proposition 3.11]{Beyrer-Pozzetti} yields the following.

\begin{prop}\label{prop.projection of cross ratio}
	Let $(V_1,W_1,W_2,V_2)\in\calA_k$ with $\dim V_1\cap V_2=k-1$. Set 
	$$V:=(V_1\cap V_2),\quad V_+:=V_1+V_2,\quad \widehat{V}_+:=\quotient{V_+}{V}, \quad \widehat{V}_{\perp}:=\quotient{V^{\perp}}{V}.$$
	Denote by $[V_1],[V_2],[W_1\cap V^{\perp}],[W_2\cap V^{\perp}]\in \Is_{1}(\widehat{V}_{\perp})$ and $[V_1],[V_2],[W^{\perp}_1\cap V_+],[W^{\perp}_2\cap V_+]\in \P(\widehat{V}_+)\simeq \R\P^1$ the associated subspaces. Then
	\begin{align*}
		\cro_{k}(V_1,W_1,W_2,V_2)& = \cro_{1}^{\widehat{V}_{\perp}}([V_1],[W_1\cap V^{\perp}],[W_2\cap V^{\perp}],[V_2])\\
		& = \cro_{1}^{\widehat{V}_+}([V_1],[W_1^{\perp}\cap V_+],[W_2^{\perp}\cap V_+],[V_2]),
	\end{align*}
	where we define $\cro^{\widehat{V}_+}_1$ on $\P(\widehat{V}_+)$ to be the usual projective cross ratio on $\R\P^1$.
\end{prop}

Martone--Zhang introduced a class of Anosov representations that satisfy a positivity condition with respect to the cross ratio. Recall that a quadruple  $(x,y,z,w)$ in $\dg$ is \emph{cyclically ordered} if it is positively or negatively oriented for the standard cyclic order on $\mathbb S^1=\dg$.

\begin{defn}[{\cite[Def. 2.25]{MZ}}]\label{d.posrat}
	A representation $\r:\G\to \PO(p,q)$ is called \emph{$k$--positively ratioed} if it is $k$--Anosov and for all cyclically ordered quadruples $(x,y,z,w)\in\dg^{[4]}$ 
	\begin{align}\label{eq.def pos ratio}
		\cro_k(x^k,y^{k},z^{k},w^k)\geq 1.
	\end{align}
\end{defn}
It is shown in \cite{MZ} that the inequalities in \eqref{eq.def pos ratio} are necessarily  strict.

\subsection{Property $H_k$}
The following transversality property of boundary maps was introduced by Labourie in his work on Hitchin representations, and was further studied in \cite{PSW1,ZZ} where relations with differentiability of boundary maps were established.
\begin{defn}[{\cite[Section 7.1.4]{Labourie-IM}}]
	A representation $\rho:\Gamma\to\PO(p,q)$ satisfies \emph{property $H_k$} for $2\leq k\leq p-1$ if it is $\{k-1,k,k+1\}$-Anosov and for all $(x,y,z)\in\dg^{(3)}$ the following sum is direct and thus equal to $\R^{p,q}$ 
	\begin{align}\label{eq property H}
		x^k+ \left(y^k\cap z^{p+q-k+1}\right)+ z^{p+q-k-1}.
	\end{align} 
A representation satisfies property $H_1$ if it is $\{1,2\}$-Anosov and for all $(x,y,z)\in\dg^{(3)}$ the following sum is direct and thus equal to $\R^{p,q}$ 
$$x^1+y^1+z^{p+q-2}.$$
\end{defn}

The transversality of the Anosov boundary maps implies that the sum in \eqref{eq property H} is direct if and only if the sum $\left(x^k\cap z^{p+q-k+1}\right)+ \left(y^k\cap z^{p+q-k+1}\right)+ z^{p+q-k-1}$ is direct. Moreover, taking the orthogonal complement, one easily checks that the sum is direct if and only if the following sum is direct
\begin{align}\label{e.Hq+p-k}
 \left(x^{p+q-k}\cap z^{k+1}\right)+ \left(y^{p+q-k}\cap z^{k+1}\right)+ z^{k-1}
\end{align}

It was shown in \cite{PSW1} that representations satisfying property $H_k$ have boundary maps with $C^1$ image, with tangent spaces prescribed by the Anosov boundary maps. 
The tangent space to $\Gr_k(\R^{p+q})$ at a subspace $V$ can be naturally identified with $\Hom(V,W)$ where $W$ is a vector subspace of dimension $(p+q-k)$ transverse to $V$. Since $\Is_k(\R^{p,q})$ is a submanifold of  $\Gr_k(\R^{p+q})$, we  identify its tangent space $T_{x^k}\Is_k(\R^{p,q})$ with a subspace of $\Hom(x^k,y^{p+q-k})$. 

\begin{prop}[{\cite[Proposition 8.11]{PSW1}}]\label{prop.1-hyperconvex for property Hk} If $\rho:\Gamma\to\PO(p,q)$ satisfies property $H_k$, then the boundary curve $\xi^k$ has $C^1$ image and the tangent space is given by
	\begin{align*}
		T_{x^k}\xi^k(\dg)=\{ \phi\in \Hom(x^k,y^{p+q-k})| x^{k-1}\subseteq \ker \phi, \;{\rm Im}\: \phi\subseteq  x^{k+1}\cap y^{p+q-k}  \}
	\end{align*}
	for any $y\neq x \in\dg$.
\end{prop}
\begin{proof}
	To deduce this statement from \cite[Proposition 8.11]{PSW1}, we set in their notation $p=k$ and $s=k+1$, and observe that  Condition (i) in \cite[Proposition 8.11]{PSW1} is guaranteed since $\rho$ is $(p-1)$-Anosov. \cite[Proposition 8.11]{PSW1} then gives that $\wedge^k \rho$ is $(1,1,2)$-hyperconvex and thus has $C^1$ 1-Anosov boundary map by \cite[Proposition 7.4]{PSW1} with tangent space given by the 2-Anosov boundary map of $\wedge^k \rho$. Since the $1$-Anosov boundary map of $\wedge^k \rho$ is given by the composition of the $k$-Anosov boundary map of $\rho$ with $\wedge^k:\Is_k(\R^{p,q})\to \P(\wedge^k \R^{p,q})$, and the latter is algebraic, we deduce that the $k$-Anosov boundary map is itself $C^1$. The statement about the tangent space, follows from the proof of \cite[Proposition 8.9]{PSW1} { where it is observed that the image of the vector subspace $$\{ \phi\in \Hom(x^k,y^{p+q-k})| x^{k-1}\subseteq \ker \phi, \;{\rm Im}\: \phi\subseteq  x^{k+1}\cap y^{p+q-k}  \}<T_{x^k}\Is_{k}(\R^{p,q})$$  under $d\wedge^k$ is the tangent space at $\xi^2_{\wedge^k\rho}(x)$ of the submanifold $\xi^2_{\wedge^k\rho}$. }
	\end{proof}

 \section{$\T$-Positive structure for $\PO(p,q)$}\label{sec.prelim positive reps}
 
In this section we recall the definition and general facts about the $\T$-positive structure, $\T$-positive triples and $\T$-positive representations associated to $\PO(p,q)$, restricting, as always, to the case $1<p< q$. We additionally prove several basic facts that will be relevant for us in the the rest of the paper.

We fix a basis $(e_1,\ldots,e_{p+q})$ such that the form $Q$ is 
represented by the matrix
\begin{align}\label{eq.the from Q}
	Q=\bpm 0&0&K\\0&J&0\\K^t&0&0\epm,
\end{align} 
where 
$$K=\begin{pmatrix}0&0&0&(-1)^{p-1}\\ 0&0&\iddots &0\\ 0&1&0 &0\\ -1&0&0&0\end{pmatrix}\text{ and } J=\begin{pmatrix}0&0&1\\ 0&-\Id_{q-p}&0\\1&0&0\end{pmatrix}.$$

An important role for the $\T$-positive structure will be played by the vector subspace $V_J:=\spa \{e_{p},\ldots,e_{q+1}\}$ together with the  bilinear form\footnote{By construction this represents the restriction of $Q$ on $V_J$.} $b_J$ of signature $(1,q-p+1)$ induced by $J$,  and the induced quadratic form $q_J$. Namely 
$$b_J(\sfw,\sfz):=\frac{1}{2} \sfw^t J \sfz,\quad q_J(\sfw):=b_J(\sfw,\sfw),\quad \sfw,\sfz\in V_J.$$
We will  denote by
\begin{equation}\label{e.cJ}
	c_J(V_J)=\{\sfw=(\sfw_{p},\ldots,\sfw_{q+1})\in V_J| \sfw_{p}>0, \;q_J(\sfw)>0\}
\end{equation}
the cone of the $Q$--positive vectors with positive first entry and by $ \ov c_J(V_J)$ its partial closure, where the second inequality isn't required to be strict. For every $\sfw$ in $c_J(V_J)$ we  have $\sfw_{q+1}\geq0$: {indeed denoting by $\ov\sfw$ the vector $(\sfw_{p+1},\ldots, \sfw_q)$ we have $$q_J(\sfw)=\sfw_p\sfw_{q+1}-\|\ov\sfw\|^2>0,$$  from which the claim follows.
}

The following basic observation will be important later on.

\begin{lem}\label{lem.pre.positiv bJ}
	Let $\sfv\in c_J(V_J)$, and $\sfw\in \ov c_J(V_J)$. Then $b_J(\sfv,\sfw)>0$.
\end{lem}

\begin{proof}
	If $\sfv\in c_J(V_J)$, and $\sfw\in \ov c_J(V_J)$, then the quadratic form $b_J|_{\langle \sfv,\sfw\rangle}$ has signature (1,1). Indeed  since $b_J$ has signature $(1,q-p+1)$ and $q_J(\sfv)>0$, the restriction of $b_J$ to $\sfv^\perp$ is negative definite.  
	
	In particular we can choose a basis of $\langle \sfv,\sfw\rangle$ with respect to which $b_J$ is represented by the matrix $\left(\begin{smallmatrix}0&1\\1&0\end{smallmatrix}\right)$. Since $\sfv,\sfw$ are (semi)-positive vectors for $b_J$, they belong either to the closure of the positive or of the negative quadrant; the {(semi)-}positivity of the last coordinates of a vector in $c_J(V_J)$ guarantees that  they furthermore belong to the same quadrant, and thus their inner product is positive.
\end{proof}

\subsection{$\T$-Positive elements}\label{ss.pos}
Guichard--Wienhard defined the notion of $\Theta$-positive structure. We introduce here the $\Theta$-positive structure of $\PO(p,q)$ following \cite[Section 4.5]{GWpositivity}.

We denote by $\Theta$ the first $p-1$ simple roots of $\PO(p,q)$ in the standard way of drawing the Dynkin diagram\footnote{Namely $\Theta$ consists of all roots but the simple root that is only connected to one other root by a double arrow; equivalently of all the long roots.}, and by $\calF_{p-1}=\calF_{p-1}(\R^{p,q})$ the associated flag manifold, which we realize as the set of flags of subspaces of $\R^{p,q}$ of dimension $\{1,\ldots, p-1,q+1,\ldots,q+p-1\}$ such that the first $p-1$ subspaces are isotropic for $Q$ and the other subspaces are their orthogonals with respect to $Q$. Clearly an element in $\calF_{p-1}$ is uniquely determined by the first $p-1$ subspaces.
Throughout the whole paper we will denote by $\sZ$ and $\sX$ the partial flags in $\calF_{p-1}$ defined by
\begin{align}\label{e.XZ}
	\sZ^l=\langle e_1,\ldots, e_l\rangle\quad \text{ and }\quad \sX^l=\langle e_{p+q},\ldots, e_{p+q-l+1}\rangle.
\end{align} 
Here, as above, $l$ ranges in the set $\{1,\ldots, p-1, q+1,\ldots, q+p\}$. Observe that $\sX^l\cap \sZ^{p+q-l+1}=\<e_{p+q-l+1}\>$. 

Given a positive real number $s$, and an integer $1\leq k\leq p-2$ the \emph{elementary matrix} $E_{k}(s)$ is the matrix that differs from the identity only in the positions $(k,k+1)$ and $(p+q-k, p+q-k+1)$ where it is equal to $s$. Instead, for $k=p-1$, we choose a vector $ \sfs\in c_J(V_J)$, denote by $\sfs^t$ the transposed vector  and  set 
\begin{equation}\label{e.Ep-1}
	E_{p-1}(\sfs)=\bpm 
	\Id_{p-2}&0&0&0&0\\
	0&1& \sfs^t&q_J(\sfs)&0\\
	0&0&\Id_{q-{p+2}}&J \sfs&0\\
	0&0&0&1&0\\
	0&0&0&0&\Id_{p-2}
	\epm.
\end{equation}

Our next goal is to recall the definition of the \emph{$\T$-positive semigroup} $U_\Theta^{>0}$ of the unipotent subgroup $U_{\Theta}$ of the stabilizer in $\PO(p,q)$ of the partial flag $\sZ$ (cfr. \cite[Theorem 4.5]{GWpositivity}). To this aim we consider the cone
$$V_{\Theta}:=\left.\{ v=(s_1,\ldots,s_{p-2},s_{p-1}^t)^t\in \R^q\right| s_1,\ldots, s_{p-2}\in \R_{>0}, s_{p-1} \in c_J(V_J)\}.$$
Given $v\in V_{\Theta}$ we set 
$$a(v)=\left(\prod_{j\leq p-1,\: j \text{ odd}} E_j(s_j)\right) \quad b(v)=\left(\prod_{j\leq p-1,\: j \text{ even}} E_j(s_j)\right).$$
Since the matrices in the product defining $a$,  respectively  $b$, commute, their order plays no role.
It will be convenient to also use the notation $ab(v):=a(v)b(v)$.
\begin{defn}\label{d.pos}
	Let $ \vec v=(v_1,\ldots  ,v_{p-1})\in V_{\Theta}^{p-1}$. The \emph{$\T$-positive element} $P( \vec v)$ is the product
	$$P(\vec v)=ab(v_1)\ldots ab( v_{p-1})$$
\end{defn}
The \emph{$\T$-positive semigroup} $U_\Theta^{>0}$ is the set of the $\T$-positive elements defined above, which  forms a semigroup  
\cite[Corollary 8.16 (2)]{GWPaper}.\footnote{In order to match this with the discussion in \cite[Definition 6.5 ]{GWpositivity} 
	it is useful to know that the Coxeter number of the root system of type $B_{p-1}$, $p\geq 3$ is equal to $2(p-1)$ \cite[pp.150-151]{Bou2}. Furthermore if $W_{B_{p-1}}$ denotes the Weyl group associated to the root system $B_{p-1}$, we choose  the standard generating system $S$ of $W_{B_{p-1}}$ given by $S= \{s_1,\ldots,s_{p-1}\}$, where $s_{p-1}$ corresponds to the reflection along the only short root in a set of simple roots. Then  the longest element $w_0$ of $W_{B_{p-1}}$ can be expressed as $w_0=(a b)^{p-1}$, where  $a$ is the product of all the elements of $S$ with odd index and  $b$ is the product of all elements of $S$ with even index, and the latter is a reduced expression \cite[pp.150-151]{Bou2} (see also \cite[Lemma 4.3]{DS}), such reduced expression is different from the one chosen in \cite[Appendix B]{GWPaper}, but it is proven in \cite[Theorem 8.1 (8)]{GWPaper} that the positive semigroup doesn't depend on the choice of the reduced expression.}

\begin{remark}
	There are $2^{p-1}$ possible different choices for the cone $V_\Theta$, depending on the choice of the sign of each of the first $p-2$ entries and of the sign of the first coordinate of the vector $s_{p-1}$. Each such choice gives a different choice of a $\T$-positive semigroup and any two choices are conjugated in $\PO(p,q)$.
\end{remark}
\subsection{$\Theta$-positivity of triples of partial flags}
We can now define positivity for triples and $n$-tuples of flags associated to $\PO(p,q)$ (cfr. \cite[Definition 4.6]{GWpositivity}).
\begin{defn}
	A triple $(x,y,z)\in \calF_{p-1}^3$ is $\Theta$-\emph{positive} if there exists $g\in\PO(p,q)$ and a $\T$-positive element $P\in U_{\T}^{>0}$ such that 
	$$(gx,gy,gz)=(\sX, P\sX,\sZ).$$
	A $n$-tuple $(x_1,\ldots, x_n)\in \calF_{p-1}^n$ is $\T$-\emph{positive} if there exist $P_2,\ldots, P_{n-1}\in  U_{\T}^{>0}$, 
	 and $g\in\PO(p,q)$ such that
	\begin{align}\label{e.posntuple}
		(gx_1,gx_2,\ldots, gx_{n-1},gx_n)=(\sX, P_2\sX,\ldots, P_2\cdots P_{n-1}\sX,\sZ).
	\end{align}
A \emph{standard $\T$-positive $n$-tuple} is an $n$-tuple of the form $$(\sX, P_2\sX,\ldots, P_2\cdots P_{n-1}\sX,\sZ) \quad \text{ for }P_2,\ldots, P_{n-1}\in  U_{\T}^{>0}.$$ 
\end{defn}
\noindent Since $U^{>0}_{\Theta}$ is a semigroup,  $P_1 P_{2}\in U^{>0}_{\Theta}$ for $P_1, P_{2}\in U^{>0}_{\Theta}$.

{ In order to obtain a notion that is invariant by the $\PO(p,q)$-action by conjugation it is necessary to consider as $\T$-positive also negatively oriented triple: already in the case  of $\PO(1,2)=\Isom(\mathbb H^2)$  both positively and negatively oriented triples in the circle belong to the same $\PO(1,2)$-orbit. In general we have:}

\begin{prop}[{\cite[Proposition 10.15 (3)]{GWPaper}}]\label{c.permutation}
	If $(x_1,x_2,x_3)\in\calF_{p-1}^{3}$ is $\T$-positive, then for any permutation $\sigma$ the triple $(x_{\sigma(1)},x_{\sigma(2)},x_{\sigma(3)})$ is positive.
\end{prop}
{In order to choose a coherent orientation for various triples one uses positivity of $n$-tuples for $n\geq 4$} 
\begin{prop}[{\cite[Lemma 10.24]{GWPaper}}]\label{c.npermutation}
	For every $n$, any  $\T$-positive  $n$-tuple $(x_1,\ldots ,x_n)\in\calF_{p-1}^{n}$, and any  permutation $\sigma$ in the $n$-th Dihedral group, the triple $(x_{\sigma(1)},\ldots ,x_{\sigma(n)})$ is positive. 
\end{prop}
We conclude the subsection with another useful characterization of positive triples proven by Guichard-Wienhard. Given a flag $A\in\calF_{p-1}$ we denote by $\Omega_A$ the set of flags transverse to $A$, namely the set of flags $F\in\calF_{p-1}$ such that for every $1\leq k\leq p-1$, the sum $A^k+(F^k)^\perp$ is direct.

\begin{thm}[{\cite[Theorem 9.2]{GWPaper}}]\label{t.ccomp}
	The set 
	$$\{F\in\calF_{p-1}|\;F =P\sX, P\in U_{\T}^{>0} \}$$
	is a connected component of $\Omega_{\sX}\cap\Omega_{\sZ}$.
	Thus, if $A,B$ are transverse flags, the set 
	$$\{F\in\calF_{p-1}|\;(A,F,B) \text{ is $\T$-positive }\}$$
	is a union of connected components of $\Omega_A\cap\Omega_B$.
\end{thm}
We record the following useful corollary.
\begin{cor}\label{lem.deform positive tuples}
	Let $c:[0,a]\subset \R\to \calF^n_{p-1}(\R^{p,q}),\; c(t)=(x_1(t),\ldots,x_n(t))$ be a continuous path such that $x_{i}(t)\tv x_j(t)$ for $i\neq j$ and all $t\in [0,a]$. If $c(0)$ is a $\T$-positive $n$--tuple, then $c(t)$ is a $\T$-positive $n$--tuple for all $t\in [0,a]$.
\end{cor}

\begin{proof}
	Let $g\in \PO(p,q)$ with $g c(0)=(\sX,P_2\sX,\ldots, P_2\cdots P_{n-1}\sX,\sZ)$.  Since the identity component $\PO_\circ(p,q)$ acts transitively on transverse pairs, we find a continuous curve $t\mapsto g_t\in \PO(p,q)$ such that $g_tx_1(t)=\sX$, $g_t x_n(t)=\sZ$ and $g_0=g$. Thus for all $t\in [0,a]$, the point $g_tx_2(t)$ is in the same connected component of $\Omega_{\sX}\cap\Omega_{\sZ}$ as $P_2\sX$. Hence we find a continuous map $t \mapsto P_2(t)\in U^{>0}_{\Theta}$ such that $g_tx_2(t)=P_2(t) \sX$. By  the same reasoning we find a continuous map $t \mapsto P_3(t)\in U^{>0}_{\Theta}$ such that $(P_2(t)^{-1}g_t x_2(t), P_2(t)^{-1}g_t x_3(t), g_t x_n(t))=(\sX,P_3(t)\sX, \sZ)$, i.e. $$(g_tx_1(t),g_t x_2(t), g_t x_3(t), g_t x_n(t))=(\sX,P_2(t)\sX,P_2(t)P_3(t)\sX, \sZ).$$
	The claim follows by induction. 
\end{proof}
\subsection{$\T$-Positive triples in $\PO(2,q)$}\label{s.2qpositive}
In the case $p=2$ the flag manifold $\calF_{1}=\Is_1(\R^{2,q})$ is also known as the Einstein universe (see for instance \cite{EinPrimer} for an introduction to this space). In this case a transverse pair $(x,z)\in\Is_1(\R^{2,q})$ determines a positive cone $C_{x,z}^+\subset\Is_1(\R^{2,q})$; this is the unique connected component   of the set of points transverse to both $x$ and $z$ whose intersection with a small neighbourhood of $x$ is in the future of $x$ and  similarly it is locally in the past of $z$. The pair $(x,z)$ also  determines a negative cone $C_{x,z}^-=C_{z,x}^+\subset\Is_1(\R^{2,q})$. Note that the choice of sign depends on a time orientation of the Einstein universe, and is thus only invariant by an index 2 subgroup of $\PO(2,q)$, while there are elements in  $\PO(2,q)$ exchanging $C_{x,z}^+$ and $C_{x,z}^-$. A triple $(x,y,z)$ is $\T$-positive if $y$ belongs either to the positive or to the negative cone. If a $4$-tuple $(x,y,z,w)$ is positive then necessarily not only $y$ and $z$ belong to the same of the two cones determined by $x,w$, say $C_{x,w}(z)$, but also the containment $C_{y,w}(z)\subset C_{x,w}(z)$ holds. 

We will not need the interpretation in terms of the Einstein universe, but several geometric properties of the positive semigroup $U_\T^{>0}$ and of positive $n$-tuples in $\Is_1(\R^{2,q})$ which we collect in this section.
}\begin{example}[Case $p=2$]\label{e.3.9}
	Any element $P\in U_{\T_2}^{>0}$ can be written as
	\begin{align}\label{eq.explicit u(w) for p=2}
		P(\sfs)=E_1(\sfs)=\bpm 1 & \sfs^t & q_J(\sfs)\\ 0 & \Id_{q} & J\sfs\\ 0& 0&1\epm,
	\end{align}
	for some $\sfs\in V_{\T_2}= c_J(V_J)$.
\end{example}
For the sake of readability we denote by $e_i$ also the line generated by the $i$-th basis vector.
Any element $x\in \calF_1(\R^{2,q})=\Is_1(\R^{2,q})$ transverse to $e_1$ has a unique representative of the form $(q_J(\sfs_x),\sfs_x,1)$.
The following elementary lemma will be useful in the proof of Proposition~\ref{p.collar2,n}. 
\begin{lem}\label{l.fourpos}
	The 4-tuple $(e_{q+2}, x, y, e_1)\in\calF^4_1(\R^{2,q})$ is equal to a standard $\T$-positive 4-tuple $(\sX, P(\sfs_1)\sX, P(\sfs_1)P(\sfs_2)\sX, \sZ)$ for some $\sfs_1,\sfs_2\in U_{\Theta_2}^{>0}$ if and only if both $\sfs_x$ and $\sfs_y-\sfs_x$ are in $c_J(V_J)$, i.e. are positive for $q_J$ and have positive first entry.
\end{lem}
\begin{proof}
	Observe that $e_{q+2}=\sX$ and $e_1=\sZ$. Moreover $x=E_1(J\sfs_x)\sX$ and $y=E_1(J\sfs_y)\sX=E_1(J\sfs_x)E_1(J\sfs_y-J\sfs_x)\sX$. Setting $\sfs_1=J\sfs_x$ and $\sfs_2=J(\sfs_y-\sfs_x)$, the claim follows from the fact that $J$ preserves $c_J(V_J)$.
\end{proof}

As a result we obtain the following compatibility of cross ratio and positivity, which will be useful later.
\begin{prop}\label{p.poscr1}
	If a 4-tuple $(a,b,c,d)\in\Is^4_1(\R^{2,q})$ is positive, then 
	$$\cro_{1}(a,b,c,d)>1.$$
\end{prop}
\begin{proof}Since both notions are invariant by the $\PO(2,q)$ action, we can assume that $(a,b,c,d)=(e_{q+2}, x, y, e_1)$ is a standard $\T$-positive 4-tuple. A direct computation gives 
	$$\cro_{1}(e_{q+2}, x, y, e_1)=\frac{q_J(\sfs_y)}{q_J(\sfs_x)}=\frac{q_J(\sfs_x)+q_J(\sfs_y-\sfs_x)+2b_J(\sfs_x,\sfs_y-\sfs_x)}{q_J(\sfs_x)}.$$
	The result follows from the definition of $c_J(V_J)$, which ensures that $q_J(\sfs_y-\sfs_x)$ is positive, and Lemma~\ref{lem.pre.positiv bJ} which ensures that $2b_J(\sfs_x,\sfs_y-\sfs_x)$ is positive.
\end{proof}	

{Given an element $g\in \PO(2,q)$ we denote by  $g^t$ its transpose, namely the element represented by the transposed matrix with respect to the standard basis. By our choice of quadratic form $g^t\in \PO(2,q)$ as well. Observe that for every $P\in\ U_{\Theta}=\Stab( \sZ)$, $P^t\in\Stab( \sX)$.}
\begin{lem}\label{l.transpose}
	For any $P\in U_{\Theta}^{>0}$ there is $Q\in U_{\Theta}^{>0}$ such that $P\sX=Q^t\sZ$ and $P^{-1}\sX=(Q^t)^{-1}\sZ$. Conversely for any   $Q\in U_{\Theta}^{>0}$ there is $P\in U_{\Theta}^{>0}$ such that $Q^t\sZ=P\sX$ and $(Q^t)^{-1}\sZ=P^{-1}\sX$. 
\end{lem}
\begin{proof}
	This is a direct computation. For example if $P=P(v)$ for some $v\in c_J(V_J)$ then $Q=P(q_J(v)^{-1}Jv)$. Since $P^{-1}= P(-v)$, we also get $P^{-1}\sX= P(-q_J(v)^{-1}Jv)^t \sZ$, as desired. 
\end{proof}

\begin{lem}\label{l.trival-tangent}
	Given $Q\in U_{\Theta}^{>0}$. Then $Q^t$ acts trivially on $T_{\sX}\Is_1(\R^{2,q})$.
\end{lem}
\begin{proof}
	We identify $T_{\sX}\Is_1(\R^{2,q})$ with $V_J$ via the inverse of the map $v\mapsto \left.\frac{d}{ds}\right|_{s=0} P(sv)\sX$. We compute
	$$P(w)^t P(sv)\sX= 
\begin{pmatrix} \frac{s^2q_J(v)}{s^2q_J(v)q_J(w)+s(v^tw)+1}\\ \frac{s^2q_J(v)w + sJv}{s^2q_J(v)q_J(w)+s(v^tw)+1} \\ 1   
	\end{pmatrix}
.$$
Taking $\left.\frac{d}{ds}\right|_{s=0}$ gives the desired result.
\end{proof}
}

\subsection{$\T$-Positive representations and $\T$-positive curves}\label{s.posproj}

\begin{defn}[{\cite[Definition 5.3]{GWpositivity}}]
	A map $\xi:\mathbb S^1\to \calF_{p-1}$ is \emph{$\T$-positive} if, for every positively oriented $n$-tuple $(x_1,\ldots,x_n)\in\mathbb{S}^1$, the $n$-tuple $(\xi(x_1),\ldots,\xi(x_n))\in\calF_{p-1}^n$ is $\T$-positive. 
\end{defn} 

{{The following is an immediate consequence of Corollary~\ref{lem.deform positive tuples} and the fact that the set of positively oriented $n$--tuples in $(\mathbb{S}^1)^{(n)}$ is connected.
		
		\begin{cor}\label{cor.pos tuples}
			Let $\xi:\mathbb{S}^1\to \calF_{p-1}$ be a continuous transverse curve. If the image of one positively oriented $n$-tuple in $ \mathbb{S}^1$ under $\xi$ is a $\T$-positive $n$-tuple in $\calF_{p-1}^n$, then every positively oriented $n$-tuple in $ \mathbb{S}^1$ is mapped to a $\T$-positive $n$-tuple in $\calF_{p-1}^n$.
		\end{cor}
		
		In particular such a map $\xi$ is $\T$-positive if and only if for every $n\geq 3$ there is one positively oriented $n$-tuple in $ \mathbb{S}^1$ which is mapped to $\T$-positive $n$-tuple in $\calF_{p-1}^n$.}}

\begin{defn}[{cfr. \cite[Definition 5.3]{GWpositivity}}]
	A  representation $\rho:\Gamma\to \PO(p,q)$ is \emph{$\T$-positive Anosov} if it is $\T$-Anosov and the Anosov boundary map $\xi:\dg\to \calF_{p-1}$ is a $\T$-positive map. We will sometimes write just \emph{positive Anosov representation}.\footnote{In the original definition of $\Theta$-positive representations there is no requirement on the representation being Anosov. In the work of Guichard--Labourie--Wienhard \cite[Proposition 5.8]{GLW}, appearing at the same time as this paper, the authors show that any $\T$-positive representation is $\T$-positive Anosov. It is not immediate that a representation that is $\Theta$-positive and $\Theta$-Anosov is $\Theta$-positive Anosov according to our definition, as, a priori, the $\Theta$-positive, and the $\Theta$-Anosov boundary map might be different: the $\Theta$-positive boundary map might, a priori, not be dynamics preserving.}
\end{defn}

This is a conjugation invariant notion: the conjugate of a $\Theta$-Anosov representation is $\Theta$-Anosov, and the conjugate of a  $\Theta$-positive representation is $\Theta$-positive.

Important  examples of $\Theta$-positive Anosov representations are  representations in Fuchsian loci, namely the  subloci of the set of $\T$-positive Anosov representations arising through the following construction:
\begin{example}[cfr. {\cite[Section 7]{ABC-IM}}]\label{e.positive fuchsian}
	Let $\R^{p,p-1}\oplus \R^{q-p+1}=\R^{p,q}$ be an orthogonal splitting of $\R^{p,q}$,  denote by  $ \OO(p,p-1)\times \OO(q-p+1)\subset \OO(p,q)$ the subgroup preserving this splitting. For every   irreducible representation $\tau:\SL(2,\R)\to \SO(p,p-1)$, every discrete and faithful representation $\iota:\G\to \SL(2,\R)$   and  any representation $\alpha:\G\to \OO(q-p+1)$, the projectivization
			$$\r:\G\to\PO(p,q),\quad \r:=(\tau\circ\iota)\oplus\alpha$$
			is $\T$-positive Anosov.
	More generally $\tau\circ\iota$ can be replaced by any Hitchin representation $\eta:\G\to \SO(p,p-1)$.
\end{example}
 Combining Theorem~\ref{thm.C-intro} and the classification of connected components of the $\SO(p,q)$-character variety acheived in \cite{ABC-IM} we will be able to deduce that  if $q\neq p+1$, every $\T$-positive Anosov representation  into $\SO(p,q)$ can be deformed to such a representation  (Remark~\ref{rem.ht-components}) - cfr. Remark~\ref{rem.pos general groups} for $\T$-positive representations in $\OO(p,q)$ and $\SO(p,q)$.

\subsection{The boundary map of a $\Theta$-positive representation}\label{s.positive-bdy-map}
The Anosov boundary map of a $\T$-positive Anosov representation satisfies remarkable additional properties, that were proven by Sambarino, Wienhard and the second author and will be useful in our work as well.

We first consider the $k$--boundary maps for $k<p-1$. The following was shown in the proof of \cite[Theorem 10.3]{PSW2}. We include the simple argument here for the reader's convenience; this step of the proof of \cite[Theorem 10.3]{PSW2} does not require that the $\T$-positive curve is equivariant with respect to some representation.

{	Let $z\in\calF_{p-1}$ and $k<p-1$. We set
	$$\Is_{k}^{\tv z}:=\{y^{k}\in \Is_{k}(\R^{p,q})|\dim( (y^{k})^{\perp}\cap z^{k+1})=1,\; (y^{k})^{\perp}\cap z^{k-1}=\{0\}\}.$$	 
 We then have a well defined projection 
	$$\begin{array}{crcl}
		\pi^k_z:&\Is_{k}^{\tv z}&\to &\P\left(\quotient{z^{k+1}}{z^{k-1}}\right)\simeq \R\P^1,\\ &y^{k}&\mapsto &[((y^{k})^{\perp}\cap z^{k+1})+z^{k-1} ].
	\end{array}$$ 
	
		This yields a well defined projection $\pi_z^k(x)$ for any $x\in \calF_{p-1}$ transverse to $z\in \calF_{p-1}$ by restricting to $x^k$.
	
	\begin{prop}[cfr. {\cite[Theorem 10.3]{PSW2}}]\label{lem.kth-projection}
		For every $\T$-positive $(n+1)$-tuple $(x_1,\ldots,x_n,z)\in\calF_{p-1}^{(n+1)}$, and for every $k<p-1$, the $(n+1)$-tuple $(\pi_z^k(x_1),\ldots,\pi_z^k(x_n),[z^k])$ is cyclically ordered in $\P\left(\quotient{z^{k+1}}{z^{k-1}}\right)\simeq \R\P^1$.
	\end{prop}

	\begin{proof}
		This is a consequence of the explicit expression of a standard $\T$-positive triple recalled in Section~\ref{ss.pos}: we can assume without loss of generality that $(x_1,\ldots,x_n,z)=(\sX, P(\vec v_2)\sX,\ldots, P(\vec v_2)\cdots P(\vec v_n)\sX,\sZ)$ for $\T$-positive elements $P( \vec v_i)$, $\vec v_i\in V_{\T}^{p-1}$ for $i=2,\ldots,n$. Then, since $P( \vec v)$ is an element of the stabilizer of $\sZ$, we get
		\begin{eqnarray}
			\nonumber \left( \left(P(\vec v)\sX\right)^{q+p-k}\cap \sZ^{k+1}\right) +\sZ^{k-1}\\
			\nonumber =P(\vec v) \left(\sX^{q+p-k}\cap \sZ^{k+1}\right) +\sZ^{k-1}.
		\end{eqnarray}
	In particular it is enough to verify that the $k$-th coordinate of the vector $P(\vec v_2)e_{q+p-k}$ is positive and that the $k$-th coordinate of $P(\vec v_2)\cdots P(\vec v_i)\sX$ is bigger than the one of $P(\vec v_2)\cdots P(\vec v_{i+1})\sX$. Recalling the definitions one directly computes that $k$-th coordinate of the vector $P(\vec v)e_{q+p-k}$ is $\sum_{i=1}^{p-1}s_k^i >0$, where we set $\vec v=(v_1,\ldots v_{p-1})$ and $v_i=(s^i_1, \ldots, s^i_{p-1})$ and the $k$-th coordinate of $P(\vec v)P(\vec w)e_{q+p-k}$ is the sum of those of $P(\vec v)e_{q+p-k}$ and $P(\vec w)e_{q+p-k}$. Inductively the claim follows.
	\end{proof}
	
As a consequence, we obtain:
\begin{cor}[{\cite[Theorem 10.3]{PSW2}}]\label{prop.Hk for pos}
	Let $\xi:\mathbb{S}^1\to \calF_{p-1}$ be a $\T$-positive map. Then $\xi$ satisfies property $H_k$ for $k<p-1$, i.e. the transversalities of Equation \eqref{eq property H} are satisfied.
\end{cor}
\begin{proof}
This is a direct consequence of Lemma~\ref{lem.kth-projection} and Equation \eqref{e.Hq+p-k}.
\end{proof}	
}

{
	For every $z\in\calF_{p-1}$, the quotient $\quotient{z^{q+2}}{z^{p-2}}$ is naturally endowed with a bilinear form of signature $(2,q-p+2)$. Our next goal is to define a natural projection from a big subset of $\Is_{p-1}(\R^{p,q})$ to $\Is_1(\quotient{z^{q+2}}{z^{p-2}})$. 
	 Observe first that $z^{p-1}$ induces a point $[z^{p-1}]\in \Is_{1}(\quotient{z^{q+2}}{z^{p-2}})$. 
	We  denote by $\Is_{p-1}^{\tv z}$ the subset of $\Is_{p-1}(\R^{p,q})$ consisting of subspaces in general position with $z^{p-2}$: 
	$$\Is_{p-1}^{\tv z}:=\{y^{p-1}\in \Is_{p-1}(\R^{p,q})| \dim (y^{p-1}\cap z^{q+2})=1,\; y^{p-1}\cap z^{p-2}=\{0\}\}.$$	 
 We then have a well defined projection 
	$$\begin{array}{crcl}
		\pi_z^{p-1}:&\Is_{p-1}^{\tv z}&\to &\Is_1\left(\quotient{z^{q+2}}{z^{p-2}}\right),\\ &y^{p-1}&\mapsto &[(y^{p-1}\cap z^{q+2})+z^{p-2} ].
	\end{array}$$ 
For the projection $\pi_z$ we have: 	
	\begin{prop}[cfr. {\cite[Proposition 10.5]{PSW2}}]\label{prop.projections are positive}
	If the $(n+1)$-tuple $(x_1,\ldots,x_n,z)\in\calF_{p-1}^{(n+1)}$ is $\T$-positive, then $$(\pi_z^{p-1}(x_1),\ldots,\pi_z^{p-1}(x_n),[z^{p-1}])\in\Is_1(\quotient{z^{q+2}}{z^{p-2}})^{(n+1)}$$ 
	is $\T$-positive.
	\end{prop}
	\begin{proof}
 It follows from Equation \eqref{e.posntuple} that we can assume, up to conjugating with a suitable element $g\in\PO(p,q)$, that there exist $\vec{v}_2$,$\ldots,\vec{v}_{n}\in V_{\Theta}^{p-1}$ such that
		$$(x_1,\ldots,x_n,z)=(\sX, P(\vec{v}_2)\sX,\ldots, P(\vec{v}_2)\cdots P(\vec{v}_{n})\sX,\sZ).$$

		For $1\leq k\leq p-2$ the elementary matrices $E_k(s)$ act by the identity on $\quotient{\sZ^{q+2}}{\sZ^{p-2}}$, while the elementary matrix $E_{p-1}(\sfs)$ acts by the corresponding $\T$-positive element (as in Example~\ref{e.3.9}), which we still denote, with a slight abuse of notation, by  $E_{p-1}(\sfs)$. Given $\vec{v}\in V^{p-1}_{\Theta}$ we denote by $\vec{v}^{p-1}=\sum_{i=1}^{p-1} \sfs_{p-1}^i\in c_J(V_J)$. It then follows from the definition that the $(n+1)$-tuple $(\pi_z^{p-1}(x_1),\ldots, \pi_z^{p-1}({x_n}), [z^{p-1}])$ is  given by 
		$$([e_{q+2}], P(\vec{v}^{p-1}_2)[e_{q+2}],\ldots, P(\vec{v}^{p-1}_2)\cdots P(\vec{v}_{n}^{p-1})[e_{q+2}],[e_{p-1}]),$$
		which concludes the proof.
	\end{proof}	

}

Let now $\r:\G\to \PO(p,q)$ be $\{p-2,p-1\}$--Anosov.  For every $x\in\dg$ we can use the projection $\pi_{x_\rho}$ to  define  a curve $\xi_x:\dg\to \Is_{1}(\quotient{x_\rho^{q+2}}{x_\rho^{p-2}})$ by 
\begin{equation}\label{e.xix}
	\begin{cases}
		\xi_x(y)&:= \pi_{x_\rho}^{p-1}(y_\rho^{p-1})\\
		\xi_x(x)&:= [x_\rho^{p-1}].
	\end{cases}
\end{equation}
An immediate consequence of Proposition~\ref{prop.projections are positive} is then
	\begin{cor}[{\cite[Proposition 10.5]{PSW2}}]\label{cor.projections are positive}
	Let $\r:\G\to\PO(p,q)$ be $\T$-positive Anosov with boundary map $\xi:\dg\to \calF_{p-1}$. Then $\xi_z:\dg\to \Is_{1}(\quotient{z_\rho^{q+2}}{z_\rho^{p-2}})$ 
	is a $\T$-positive curve for every $z\in\dg$.
\end{cor}

The following regularity properties of the map $\xi$ was obtained in \cite{PSW2} as a consequence of  Corollary~\ref{prop.Hk for pos} and Corollary~\ref{cor.projections are positive}. We  record it here for future reference:

\begin{thm}[{\cite[Corollary 10.4 and Proposition 10.5]{PSW2}}]\label{t.Lipschitz}
	Let $\rho:\Gamma\to \PO(p,q)$ be \emph{$\T$-positive Anosov}. Then the curve $\xi^{p-1}:\dg\to \Is_{p-1}(\R^{p,q})$ has Lipschitz image, and the curves $\xi^{k}:\dg\to \Is_{k}(\R^{p,q})$ for $k\leq p-2$ have $C^1$ image.
\end{thm}
In particular the image of $\xi$ is a Lipschitz submanifold of $\calF_{p-1}$, since it is the graph of monotone functions between Lipschitz circles.

\section{$\Theta$-positive representations are positively ratioed}\label{sec:pos-ratio}

{{In this section we describe tangent cones defined by $\T$-positive triples in $\calF_{p-1}$ that naturally come from the $\T$-positive structure. The main technical result of this section is that the cross ratio increases along these tangent cones (Proposition~\ref{prop.derivative cross ratio for positive triple: general}). This implies that $\Theta$-positive Anosov representations into $\PO(p,q)$ are $k$--positively ratioed for $k=1,\ldots,p-1$ (Theorem~\ref{thm.pos ratioed}).  Proposition~\ref{prop.derivative cross ratio for positive triple: general} will also be crucially used in the proof of the collar lemma.}}

The proof of  Proposition~\ref{prop.derivative cross ratio for positive triple: general} will build upon the Lipschitz regularity of the image of the boundary maps (Theorem~\ref{t.Lipschitz}). As a result, for every $k$, we can (and will) choose a Lipschitz structure on $\dg\simeq\mathbb S^1$ defining a Lebesgue measure class with the property that for almost every point $x\in\dg$, the curve $\xi^{k}$ has a non-zero derivative $\dot\xi^{k}(x)\in T_{\xi^k(x)}\Is_{k}$. The Lipschitz structures on $\dg$ induced by the boundary maps $\xi^k$ are, except for very special cases, mutually singular.  

In our analysis we will need a more precise understanding of where these derivatives lie, this is provided in the next subsection. In Subsection~\ref{s.5.2} we will  use this information to show that the derivative of the cross ratio is positive {along explicit paths that we construct with the given derivative}. Thanks to Rademacher's theorem, this is enough to conclude the main result of this section (Theorem~\ref{thm.pos ratioed}).
\subsection{Tangent cones}
The goal of this subsection is to define, for every $\T$-positive triple $(x,y,z)\in\calF_{p-1}^3$ and every $1\leq k\leq p-1$, a cone $c_k^+(x,y,z)$  in a linear subspace of $T_{x^k}\Is_k(\R^{p,q})$
with the property that, for each equivariant $\T$-positive curve $\xi$ containing $(x,y,z)$ in its image and differentiable at $x$, one has $\dot\xi^k(x)\in  c_k^+(x,y,z)$. We will also show that, if $x,y,z$ belong to the image of a positive curve $\xi:\dg\to\calF_{p-1}$, the cone $c_k^+(x,y,z)$ only depends on the point $x$, and on an orientation of $\dg$.

We focus first on {\bf the case $p=2$}, so that $\calF_1(\R^{2,q})=\Is_1(\R^{2,q})$. Given a $\T$-positive triple $(x,y,z)$, we choose $g\in\PO(2,q)$ such that $(gx,gy,gz)=(\sX,P \sX, \sZ)$ for some $\T$-positive element $P\in U_{\T}^{>0}$.
{The unipotent radical $U_\T$ of the stabilizer of $\sZ$ acts freely and transitively on the open subset of $\calF_{p-1}$ consisting of elements transverse to $\sZ$; through this action we obtain an open chart around the point $\sX$ identified with $U_\T$. Differentiating this chart and composing with the differential action of $g^{-1}$ we obtain a linear isomorphism between the Lie algebra $\frak u_\Theta$ of $U_\T$ and $T_{x^1}\Is_1(\R^{2,q})$.} 
Recall from Section~\ref{sec.prelim positive reps} that $\frak u_\Theta$ consists of matrices of the form
$$\left\lbrace \left.\bpm 0 & \sfw^t & 0\\ 0 & 0 & J\sfw\\ 0& 0&0 \epm \right|  \sfw\in\R^q\right\rbrace,$$
and thus we obtain an isomorphism $f_g:T_{x^1}\Is_1(\R^{2,q})\to V_J$. Recalling Equation \eqref{e.cJ} in Section~\ref{sec.prelim positive reps}, we define
\begin{equation}\label{e.cone}
	c_1^{+}{(x,y,z)}:= f_g^{-1}( \ov c_J(V_J)).
\end{equation}
{
		In the interpretation of $\Is_1(\R^{2,q})$ as the Einstein universe from Section~\ref{s.2qpositive},  the cone $c_1^{+}{(x,y,z)}$ is the tangent at $x$ to the cone $C_{x,z}(y)$. {The following lemma should be clear for the experts in Einstein geometry, we include a proof in our framework.}

\begin{lem}\label{l.conewdef1}\
	\begin{enumerate}
	\item The cone $c_1^{+}(x,y,z)$ doesn't depend on the element $g$.
	\item For every $g\in \PO(2,q)$, $c_1^{+}(x,y,z)=g_*^{-1}c_1^+(gx,gy,gz)$
	\item If $(x,y,z,w)\in\Is_1(\R^{2,q})^{(4)}$  is positive, then  $c_1^{+}(x,y,w)=c_1^{+}(x,z,w)=c_1^{+}(x,y,z)$. Furthermore $c_1^{+}(x,y,z)=-c_1^{+}(x,z,y)$.
	\end{enumerate}
\end{lem}
\begin{proof}\
	\begin{enumerate}
	\item Let $g_1,g_2$ be such that $(g_ix,g_iy,g_iz)=(\sX,P_i \sX, \sZ)$ for $\T$-positive elements $P_i\in U_{\T}^{>0}$. Then $h=g_2g_1^{-1}$ belongs to the stabilizer of the pair $(\sX,\sZ)$, and since $hP_1=P_2\in U_\T^{>0}$, $h$ preserves the cone $\ov c_J(V_J)$. This implies that $f_{g_1}^{-1}(\ov c_J(V_J))=f_{g_2}^{-1}(h^{-1}_*\ov c_J(V_J))=f_{g_2}^{-1}(\ov c_J(V_J))$.
	\item Follows directly from (1).
	\item{ 
		Pick $g\in \PO(2,q)$ such that $(gx,gy,gw,gz)=(\sX,P_2\sX,P_2P_3\sX,\sZ)$. Since $P_2P_3\in U_{\T}^{>0}$, it follows by definition that $$c_1^+(x,y,w)=f_g^{-1}( \ov c_J(V_J))= c_1^+(x,z,w).$$
		
	Next we show  that $c_1^{+}(x,y,z)=-c_1^{+}(x,z,y)$.  As by (2) $c_1^{+}(x,y,z)=g_*^{-1}c_1^+(gx,gy,gz)$ for any $g\in \PO(2,q)$ with differential $g_*$, it is enough to show that $c_1^{+}(\sX,P\sX,\sZ)=-c_1^{+}(\sX,\sZ,P\sX)$. Let $$H=\begin{pmatrix} -1 & &\\ & {\rm Id}_q & \\ & & -1\end{pmatrix}\in \PO(2,q).$$ 
	Then $H$ stabilizes $\sX$ and $\sZ$ and acts as $-{\rm Id}$ on $T_{\sX}\Is_1(\R^{2,q})$. Let $P=P(v)$ for some $v\in c_J(V_J)$, then
	$$HP(v)\sX= HP(v)H\sX=P(-v)\sX=P^{-1}\sX.$$ 
	Thus $c_1^{+}(\sX,P^{-1}\sX,\sZ)=- c_1^{+}(\sX,P\sX,\sZ)$. Let $Q\in U_{\T}^{>0}$ such that $P\sX=Q^t\sZ$, which exist by Lemma~\ref{l.transpose}. According to Lemma~\ref{l.transpose} we also have $(Q^t)^{-1}\sZ=P^{-1}\sX$. Thus 
	$$\hspace{1.2cm} (\sX,\sZ,P\sX)=(\sX,\sZ,Q^t\sZ)=Q^t\cdot (\sX,(Q^t)^{-1}\sZ,\sZ)=Q^t\cdot (\sX,P^{-1}\sX,\sZ).$$
 Therefore $c_1^+(\sX,\sZ,P\sX)=Q^t_* c_1^+(\sX,P^{-1}\sX,\sZ)$. Since the differential $Q^t_*$ acts trivially on $T_{\sX}\Is_1(\R^{2,q})$ by Lemma~\ref{l.trival-tangent}, $Q^t_*$  preserves the cone $c_1^+(\sX,P^{-1}\sX,\sZ)=- c_1^{+}(\sX,P\sX,\sZ)$. This proves the claim.
 
 Finally we show $c_1^{+}(x,y,z)=c_1^{+}(x,y,w)$. For this recall that by Proposition~\ref{c.npermutation} $(x,w,z,y)$ is also a positive quadruple. Thus by what we have shown so far
 $$\hspace{1.9cm} c_1^{+}(x,y,z)=-c_1^{+}(x,z,y)=-c_1^{+}(x,w,y)=c_1^{+}(x,y,w).\hspace{1cm}\qedhere$$}
	\end{enumerate}	
\end{proof}	
In particular if $x=\xi(a)$ is in the image of a $\T$-positive curve $\xi:\mathbb S^1\to \Is_1(\R^{2,q})$, the cone $c_1^{+}(x,\xi(b),\xi(c))$ doesn't depend on $b,c$, but only on the orientation induced on the circle by the triple $(a,b,c)$. For this reason, in this situation, we will simply write 
$$c_1^{+}(x):=c_1^{+}(x,\xi(b),\xi(c)) \quad \text{ for any $b,c$ inducing the given orientation.}$$
}
\begin{lem}\label{l.der1}
	Let $\xi:\mathbb S^1\to \Is_1(\R^{2,q})$ be $\Theta$-positive. If $\xi$ is differentiable at $x\in\mathbb S^1$ with non-zero derivative, then
	$$\dot\xi(x)\in c^+_1(x).$$
	Thus if $\xi$ is the boundary map of a $\T$-positive representation, then $\dot\xi(x)\in c^+_1(x)$ for almost all $x\in\dg\simeq \mathbb{S}^1$.
\end{lem}
\begin{proof}
	Let $x$ be a differentiability point for the map $\xi$. We complete $x$ to a positively oriented triple $(x,y,z)\in(\mathbb S^1)^3$, and choose an element $g\in\PO(2,q)$ so that $(g\xi(x), g\xi(y), g\xi(z))=(\sX,P \sX, \sZ)$. In this way we obtain a  smooth chart 
	$$\begin{array}{ccc}
		V_J&\to&U_x\subset \Is_1(\R^{2,q})\\
		\sfw&\mapsto&g^{-1}\exp(\sfw)\cdot\sX.
		\end{array}$$ 
	Working in this smooth chart, the derivative of $\xi$, where defined, is the limit of rescaled partial increments. Since each of those belong to $c_J(V_J)$ (compare Example~\ref{e.3.9}), 
	the limit belongs to the closure $\ov c_J(V_J)$. 
	It follows from Theorem~\ref{t.Lipschitz} that almost every point $x\in\dg$ is a differentiability point for $\xi$.
\end{proof}

We now turn to {\bf the general situation}. Recall that for any $x\in\calF_{p-1}$, the quotient $\quotient{x^{q+2}}{x^{p-2}}$ is naturally endowed with a form of signature $(2,q-p+2)$, and there is a natural inclusion
$$\phi_x:\Is_1\left( \quotient{x^{q+2}}{x^{p-2}}\right)\to \Is_{p-1}(\R^{p,q})$$
whose image is the analitc submanifold consisting of $(p-1)$-dimensional isotropic subspaces contained in $x^{q+2}$ and containing $x^{p-2}$ as a subspace. In order to improve readability  we will denote by $\Is_{p-1}\left( \quotient{x^{q+2}}{x^{p-2}}\right)$ the image of $\phi_x$.
 
Realizing $\Is_{p-1}(\R^{p,q})$ as a subset of the $(p-1)$--Grassmannian of $\R^{p+q}$, we obtain, for any $z^{q+1}$ transverse to $x^{p-1}$, an  inclusion
$$T_{x^{p-1}}\Is_{p-1}(\R^{p,q})<\Hom(x^{p-1},z^{q+1}).$$
The tangent space to  $\phi_x(\Is_1(\quotient{x^{q+2}}{x^{p-2}}))$ is then given by  linear maps in $T_{x^{p-1}}\Is_{p-1}(\R^{p,q})$ that vanish on $x^{p-2}$ and have image in $x^{q+2}$:
$$T_{x^{p-1}}\Is_{p-1}\left(\quotient{x^{q+2}}{x^{p-2}}\right):=\left\{\Phi\in T_{x^{p-1}}\Is_{p-1}(\R^{p,q})\left|\;\begin{array}{l} \Phi(x^{p-2})=0\\ \Phi(x^{p-1})< x^{q+2}\cap z^{q+1}\end{array}\right.\right\}.$$

Let now $(x,y,z)\in\calF_{p-1}^{(3)}$ be $\T$-positive. Recall from Proposition~\ref{prop.projections are positive} that we denote by $\pi_x^{p-1}: \Is_{p-1}^{\tv z}\to\Is_1(\quotient{x^{q+2}}{x^{p-2}})$ the natural projection, and that if $(x,y,z)\in\calF_{p-1}^{(3)}$ is $\T$-positive, then $y,z$ belong to $\Is_{p-1}^{\tv z}$ and the triple $([x^{p-1}], \pi_x^{p-1}(y),\pi_x^{p-1}(z))\in \Is_1(\quotient{x^{q+2}}{x^{p-2}})$ is $\T$-positive. In the last statement we are also using that any  permutation of a $\T$-positive triple is $\T$-positive (Corollary~\ref{c.permutation}).
\begin{defn}
	For a $\T$-positive triple $(x,y,z)\in\calF_{p-1}^{(3)}$ we set
	$$c_{p-1}^+(x,y,z):=d\phi_x(c_1^+([x^{p-1}],\pi_x^{p-1}(y),\pi_x^{p-1}(z))).$$
\end{defn}
It is clear from the construction that for every $\T$-positive triple $(x,y,z)\in\calF_{p-1}^{(3)}$ and for any $g\in\PO(p,q)$ it holds 
$$c_{p-1}^+(gx,gy,gz)=g_*c_{p-1}^+(x,y,z),$$
where $g_*$ denotes the induced action of $g$ on $T_*\Is_{p-1}(\R^{p,q})$.

Combining Lemma~\ref{l.conewdef1} and Proposition~\ref{prop.projections are positive}, which ensures that the image under $\pi_x^{p-1}$ of a positive $n$-tuple is positive, we obtain that if $x=\xi(a)$ is in the image of a $\T$-positive curve $\xi:\mathbb S^1\to \calF_{p-1}(\R^{p,q})$, the cone $c_{p-1}^{+}(x,\xi(b),\xi(c))$ doesn't depend on $b,c$, but only on the orientation induced on the circle by the triple $(a,b,c)$. As in the case $p=2$, in this situation, we will write 
$$c_{p-1}^{+}(x):=c_{p-1}^{+}(x,\xi(b),\xi(c)) \quad \text{ for any $b,c$ inducing the given orientation.}$$}
While in the case $p=2$ the cone $c_1^+(x)$ has full dimension, for $p\geq 3$ it is supported in a proper vector subspace.

In the following proposition we use the $\T$-positive curves  $\xi_x:\dg\to \Is_1(\quotient{x^{q+2}}{x^{p-2}})$ associated to a $\T$-positive curve $\xi:\dg\to \calF_{p-1}(\R^{p,q})$ through Equation \eqref{e.xix}.

\begin{prop}\label{p.p-1derivative}
	Let $\rho:\G\to\PO(p,q)$ be $\T$-positive Anosov. If the derivative $\dot\xi^{p-1}(x)$ of $\xi$ at $x$ exists, 
	we have
	$$\dot\xi^{p-1}(x)\in c^+_{p-1}(x).$$
\end{prop}
\begin{proof}
We divide the proof in two steps. First we show that if the derivative $\dot\xi^{p-1}(x)$ of $\xi$ at $x$ exists, then $\dot\xi^{p-1}(x)\in T_{x^{p-1}}\Is_{p-1}\left(\quotient{x^{q+2}}{x^{p-2}}\right)$.
	
	We denote by $\Psi\in\Hom(x^{p-1},z^{q+1})$ the vector corresponding to the derivative $\dot\xi^{p-1}(x)$. Since $\rho$ has property $H_{p-2}$, the curve $\xi^{p-2}$ is $C^1$ and  the tangent space $T_{x^{p-2}}\xi^{p-2}$ is generated by a vector 
	$\Phi\in\Hom(x^{p-2},z^{q+2})$ {with} $\im(\Phi)<x^{p-1}$ (cfr. Proposition~\ref{prop.1-hyperconvex for property Hk}). 
	In turn this means that infinitesimal variations of vectors in $x^{p-2}$ are contained in $x^{p-1}$, and thus it follows that $\Psi$ satisfies $\ker\Psi=x^{p-2}$.
	Since $\ker \Psi = x^{p-2}$, and $x^{p-1}$ is isotropic, $\im\Psi< (x^{p-2})^\bot=x^{q+2}$.
	As a result $\dot\xi^{p-1}(x)\in T_{x^{p-1}}\Is_{p-1}\left(\quotient{x^{q+2}}{x^{p-2}}\right)$.
	
	We conclude by showing that then $\dot\xi^{p-1}(x)$ necessarily belongs to the cone. By definition of the cone, we have $c_{p-1}^+(x)\subset T_{x^{p-1}}\Is_{p-1}\left(\quotient{x^{q+2}}{x^{p-2}}\right)$. To show that $\dot\xi^{p-1}(x)\in c_{p-1}^+(x)$ we can assume, up to the action of $\PO(p,q)$, that $(\xi(x),\xi(y),\xi(z))=(\sX,P \sX,\sZ)$ is a standard $\T$-positive triple. Then consider the projection 
	$$\begin{array}{crcl}
		\pi_\sZ:&\Is_{p-1}^{\tv\sZ}&\to &\Is_1\left(\quotient{\sZ^{q+2}}{\sZ^{p-2}}\right),\\ &Y^{p-1}&\mapsto &[(Y^{p-1}\cap \sZ^{q+2})+\sZ^{p-2} ],
		\end{array}$$ 
	where, as always,
	$$\Is_{p-1}^{\tv\sZ}:=\{Y^{p-1}\in \Is_{p-1}| \dim (Y^{p-1}\cap \sZ^{q+2})=1,\, Y^{p-1}\cap \sZ^{p-2}=\{0\}\}.$$ 
The set $\Is_{p-1}\left( \quotient{\sX^{q+2}}{\sX^{p-2}}\right)$ is contained in the domain of definition of $\pi_\sZ$ and  $\pi_\sZ$ induces a diffeomorphism between $\Is_{p-1}\left( \quotient{\sX^{q+2}}{\sX^{p-2}}\right)$ and $\Is_{1}\left( \quotient{\sZ^{q+2}}{\sZ^{p-2}}\right)$, so that it is enough to check that $$d\pi_\sZ(\dot\xi^{p-1}(x))\in d\pi_\sZ (c_{p-1}^+(\sX)).$$
	
	Observe that $\pi_\sZ\circ \xi =\xi_{z}$ on $\dg\backslash\{z\}$. Moreover  Proposition~\ref{prop.projections are positive} guarantees that  $(\xi_z(x),\xi_z(y),\xi_z(z))$ is a standard $\T$-positive triple in $\Is_1(\quotient{\sZ^{q+2}}{\sZ^{p-2}})$. 
Since $\xi_z$ is $\T$-positive and differentiable at $x$ (because $\xi$ is differentiable at $x$  by assumption), we deduce from Lemma~\ref{l.der1} that $$d\pi_\sZ(\dot\xi^{p-1}(x))\in c_{1}^+([\sX])=c_{1}^+(\xi_z(x)),$$ where $c_{1}^+([\sX])$ is the cone defined by the standard $\T$-positive triple in $\Is_1(\quotient{\sZ^{q+2}}{\sZ^{p-2}})$. 

We claim that $d\pi_\sZ (c_{p-1}^+(\sX))=c_{1}^+(\xi_z(x))$. Indeed 
$$\exp(c_{p-1}^+(\sX))=\{E_{p-1}(\sfs) \cdot \sX|\; \sfs\in c_J(V_J)\},$$
 where $E_{p-1}(\sfs)$ is as in Equation \eqref{e.Ep-1}. Thus 
 $$\pi_\sZ (\exp(c_{p-1}^+(\sX)))=\{E^{\sZ}_{1}(\sfs) \cdot \xi_z(x)| \sfs\in c_J(V_J)\},$$ where $E^{\sZ}_{1}(\sfs)$ is the (only) elementary matrix for the $\T$-positive structure on $\quotient{\sZ^{q+2}}{\sZ^{p-2}}$ (as in Example~\ref{e.3.9}). Since 
 $$\{E^{\sZ}_{1}(\sfs) \cdot \xi_z(x)|\; \sfs\in c_J(V_J)\}=\exp(c_{1}^+([\sX]))=\exp(c_{1}^+(\xi_z(x)))$$
  and $\exp$ is a diffeomorphism, the claim follows. 
  
  Thus, as desired, 
\[
  d\pi_\sZ(\dot\xi^{p-1}(x))\in d\pi_\sZ (c_{p-1}^+(\sX)).
\qedhere \]
\end{proof}

The same analysis can be done for the other boundary maps $\xi^k:\dg\to \Is_k(\R^{p,q})$. In this case, since the $\T$-positive curve has property $H_k$ (Corollary~\ref{prop.Hk for pos}), the curve $\xi^k$ is already $C^1$ with tangent space, at $x$, given by the line (Proposition~\ref{prop.1-hyperconvex for property Hk})
$$\dot x^k\in T_{x^k}[\quotient{x^{k+1}}{x^{k-1}}]\subset T_{x^k}\Is_{k}(\R^{p,q}).$$ 

In this case the  cone $c^+_{k}(x)$ is the ray corresponding to the positive orientation of the circle (compare \cite[Page 31]{Beyrer-Pozzetti}) induced by the choice of the triple. In particular the following holds.
\begin{prop}\label{p.kderivative}
	Let $\r:\G\to \PO(p,q)$ be $\T$-positive Anosov. Then for all $k=1,\ldots,p-2$ and all $x\in\dg$
	$$\dot\xi^{k}(x)\in c^+_{k}(x).$$
\end{prop}

\begin{remark}[Interpretations of the tangent cones in terms of positivity]\label{r.derTheta}
	A more Lie theoretic way of defining the cones $c_k^+(x)$ is as follows. Given a Lie group $\sf G$, a parabolic subgroup $\sf P_\theta$ which is conjugated to its opposite and a transverse pair $x,z\in \quotient{\sf G}{\sf P_\theta}$, the tangent space $T_x  \quotient{\sf G}{\sf P_\theta}$ is canonically identified with the Lie algebra $\frak n_{\theta}$ of the unipotent radical $\sf U_{\theta}$ of the stabilizer in $\sf G$ of the point $z$.
	
	If now the group $\sf G$ admits a $\Theta$-positive structure, then, by definition,  there are  $\sf L_\Theta^0$-invariant convex cones $c_\beta\subset\frak u_\beta$, where $\frak u_\beta\subset\frak n_\Theta$ is the sum of all the root spaces $\frak g_\alpha$ that are equal to $\beta$ modulo the span of $\Delta\setminus \T$. In particular, for every $\beta\in\Theta$,  $\frak u_\beta\subset \frak n_{\{\beta\}}=T_x\quotient{\sf G}{\sf P_{\{\beta\}}}$, and we have $c_k^+=c_{\alpha_k}$.
\end{remark}

\subsection{The derivative of the cross ratio}\label{s.5.2}
We can now compute the variation of the cross ratio along paths tangent to the cones $c^+_i$ introduced in the previous section, thus proving Proposition~\ref{prop.derivative cross ratio for positive triple: general}. We begin with an explicit computation in the case $p=2$.

\begin{lem}\label{lem.derivative cross ratio for positive triple p=2}
	Let $(x,y,z)\in \Is_1^3(\R^{2,q})$ be a $\T$-positive triple and $\Phi\in  c_1^+(x)\subset T_{x} \Is_1(\R^{2,q})$. Then
	$$  d_{x} \gcr_{1}(z,x,\cdot,y)(\Phi)> 0 .$$
\end{lem}

\begin{proof}
	Up to the action of an element in $\PO(2,q)$ 
	we can assume that $z=e_1$, $x=e_d$ and there exist  $\sfy\in c_J(V_J)$,  $\sfw\in \ov c_J(V_J)$ such that 
	\begin{align*}
		&y=\<\bpm q_J(\sfy)\\ \sfy \\ 1\epm\>,\quad  \left.\Phi=\frac{d}{dt}\right|_{t=0} x_t \quad \text{for    } x_t:=\<\bpm t^2 q_J(\sfw)\\ t\sfw,\\ 1\epm\>.
	\end{align*}
	As in the proof of Proposition~\ref{p.poscr1}, we have
	$$\gcr_{1}(z,x,x_t,y)= \frac{-q_J(\sfy)}{-t^2q_J(\sfw)+2tb_J(\sfw,\sfy)-q_J(\sfy)},$$
	and therefore
	$$d_{x} \gcr_{1}(z,x,\cdot,y)(\Phi)= \left.\frac{d}{dt}\right|_{t=0} \gcr_{1}(z,x,x_t,y)= \frac{2b_J(\sfw,\sfy)}{q_J(\sfy)}>0$$
	by Lemma~\ref{lem.pre.positiv bJ}.
\end{proof}

For $k=p-1$ we   reduce the general case to Lemma~\ref{lem.derivative cross ratio for positive triple p=2} using Proposition~\ref{prop.projections are positive}.
\begin{prop}\label{prop.derivative cross ratio for positive triple: general}
	Let $(x,y,z)\in\calF_{p-1}^{3}$ be a $\T$-positive triple. Let $\Phi_k\in c^+_k(x) \subset T_{x^k} \Is_k(\R^{p,q})$ for $k=1,\ldots,p-1$. 
	Then
	$$  d_{x^k} \gcr_{k}(z^{k},x^{k},\cdot,y^{k})(\Phi_k)> 0.$$
\end{prop}

\begin{proof} 
	As the cross ratio is invariant under the action of $\PO(p,q)$, we assume without loss of generality that  $x= \sX$,  $z=\sZ$, and $y=P \sX$ for some $P\in U^{>0}_{\T}$. For $k=1,\ldots,p-2$, and $t\in\R$ we consider the  element of $\Is_k(\R^{p,q})$
	$$x_t^k:= x^{k-1}\oplus \langle e_{p+q-k+1}+te_{p+q-k}\>,$$
	while for $k=p-1$ we consider the element of $\Is_{p-1}(\R^{p,q})$
	$$x_t^{p-1}:= x^{p-2}\oplus \<e_{q+2}+t{\sfw}\>,$$
	for some $\sfw\in\ov c_J(V_J)$ (recall that $V_J<\R^{p,q}$).

	By construction, for $1\leq k\leq p-1$, the map $t\mapsto x_t^k$ gives a curve in $\Is_k(\R^{p,q})$ whose derivative  lies in the cone $c_k^+$, and for every $\Phi_k\in c_k^+$ we can find such a  curve tangent to $\Phi_k$.

	Thus it is enough to check
	$$  \left.\frac{d}{dt}\right|_{t=0} \gcr_{k}(z^{k},x^{k},x_t^k,y^{k})> 0.$$
	Let $X_k:=\quotient{x^{k+1}}{ x^{k-1}}$ for $k\leq p-2$ and $X_{p-1}:=\quotient{x^{q+2}}{ x^{p-2}}$. We denote by $[\cdot]$ the obvious projection to $\P(X_k)$ (resp. $\Is_1(X_{p-1})$), so that
	$$[y]:=\left\{
	\begin{array}{lcl}
		[y^{p+q-k}\cap x^{k+1}]&&k\leq p-2\\
		
		[y^{p-1}\cap x^{q+2}]&&k=p-1
		\end{array}
		\right.$$
	 and the same for $[z]$.
	 Then we have 
	 $$\begin{array}{rl}  \gcr_{k}(z^{k},x^{k},x_t^k,y^{k})&= \gcr_{k}(x^{k},z^{k},y^{k},x_t^k)\\
	 	&=	\gcr_{1}^{X_k}([x],[z],[y],[x_t])\\
	 	&=	\gcr_{1}^{X_k}([z],[x],[x_t],[y]),
	 	\end{array}$$
 	where the first and last equality follow from Proposition~\ref{prop.property of grassmannian cro}~(2), and the second is a consequence of the second statement in Proposition~\ref{prop.projection of cross ratio}.

	For $k\leq p-2$, property $H_k$ (Corollary~\ref{prop.Hk for pos}) guarantees that $[z],[x],[y]$ are distinct (cfr. Equation \eqref{e.Hq+p-k}); by definition $[\Phi_k]=\left.\frac{d}{dt}\right|_{t=0} [x_t]\neq 0\in T_{[x]}X_k$ is  directed towards $[y]$, i.e. $[z],[x],[x_t],[y]$ are cyclically ordered on $\P(X_k)\simeq \R\P^1$. As $\gcr_{1}^{X_k}$ is the usual projective cross ratio on $\P(X_k)$, this is enough to guarantee the claim (compare \cite[Lemma 3.2]{Beyrer-Pozzetti}).
	
	In the case $k=p-1$, Lemma~\ref{lem.derivative cross ratio for positive triple p=2} guarantees that
	$$ \left.\frac{d}{dt}\right|_{t=0} \mkern-20mu\gcr_{1}^{X_{p-1}}([z],[x],[x_t],[y])>0,$$
	which proves the claim.
\end{proof}

\begin{thm}\label{thm.pos ratioed}
	Let $\r:\G\to\PO(p,q)$ be $\T$-positive Anosov. Then $\r$ is $k$--positively ratioed for $k=1,\ldots ,p-1$.
\end{thm}
\begin{proof}
Let $(x,y,z,w)$ be cyclically ordered on $\di\G$. Theorem~\ref{t.Lipschitz}  gives that, for all $1\leq k\leq p-1$, the curve $u\mapsto \xi^{k}(u)$ has rectifiable image. Thus, for each $k$, we can parametrize the interval $(y,z)\subset \di \G$ not containing $x,w$  by a map $t\mapsto y_t$ such that $t\mapsto \xi^k(y_t)$ is Lipschitz and $y_0=y$, $y_1=z$. Then also $f:[0,1]\to\R$, defined by $f(t):= \cro_k(x^k,y^k,y^k_t,w^k)$ is Lipschitz and thus by Rademacher's theorem and the fundamental theorem of calculus we get that 
\begin{align*}
\cro_k(x^k,y^k,z^k,w^k)&=f(1)=f(0)+\int_0^1 f'(s)ds\\
&= \cro_k(x^k,y^k,y^k,w^k)+ \int_0^1 \left.\frac{d}{dt}\right|_{t=s}\mkern-20mu\cro_k(x^k,y^k,y^k_t,w^k)ds
\end{align*}
We have $\cro_k(x^k,y^k,y^k,w^k)=1$. Using the cocycle identity, we deduce that (almost everywhere)
\begin{equation}\label{e.crpos}
	\left.\frac{d}{dt}\right|_{t=t_0}\mkern-20mu\cro_k(x^k,y^k,y^k_t,w^k)=\cro_k(x^k,y^k,y^k_{t_0},w^k)\left.\frac{d}{dt}\right|_{t=t_0}\mkern-20mu\cro_k(x^k,y^k_{t_0},y^k_t,w^k).\end{equation}
From Proposition~\ref{prop.derivative cross ratio for positive triple: general}, whose assumptions are satisfied thanks to Proposition~\ref{p.p-1derivative} (resp. Proposition~\ref{p.kderivative}), it follows that (almost everywhere) 
$$\left.\frac{d}{dt}\right|_{t=t_0}\mkern-20mu\cro_k(x^k,y^k_{t_0},y^k_t,w^k)>0.$$
For every $t_0$ the coefficient $\cro_k(x^k,y^k,y^k_{t_0},w^k)$ in Equation \eqref{e.crpos} is positive  since the 4-tuple $(x^k,y^k,y^k_{t_0},w^k)$ is in the same connected component of the domain of definition of $\cro_k$ as the 4-tuple $(x^k,y^k,y^k,w^k)$ for which $\cro_k(x^k,y^k,y^k,w^k)=1$ and the cross ratio is continuous and doesn't vanish (recall Definition~\ref{d.cr}). Hence $f'(t)>0$ almost everywhere and therefore $\cro_k(x^k,y^k,z^k,w^k)>1$, as desired.
\end{proof}


\section{Collar lemmas for $\Theta$-positive representations}
The goal of the section is to prove the collar lemma, stated as Theorem~\ref{thm.intro-collar} in the introduction. As a warm up, we give  in Section~\ref{s.c1} a complete proof in the case $p=2$, which builds on the  technical Lemma~\ref{p.collar2,n}, a  key ingredient also in the proof of the general case.  In Section~\ref{s.c2} we introduce the notion of hybrid flags, and discuss the general strategy of proof, which is carried out in the remaining three subsections. 
\subsection{Collar lemma for $\PO(2,q)$-positive representations}\label{s.c1}
In this subsection we focus on the case $p=2$; in this case the group $\PO(2,q)$ is of Hermitian type, and $\T$-positive Anosov representation $\r:\G\to\PO(2,q)$ are nothing but maximal representations as in \cite{BIW, CTT}. Our aim is to give a self contained proof of the collar lemma in this case, 
 which is considerably simpler, but already illustrates the topological input needed for the general proof, and sheds some light to the general strategy. Recall from the introduction that we say that two elements $g,h\in\G$ are \emph{linked} if the attracting and repelling fixed points $h_{+},h_{-}\in \dg$ of $h$ are in different connected components of $\dg\backslash\{g_-,g_+\}$. Theorem~\ref{thm.intro-collar} in this context restates as

\begin{thm}\label{t.collarmax}
		Let $\r:\G\to\PO(2,q)$ be  $\T$-positive Anosov  and $g,h\in\G$ be a linked pair. Then 
	\begin{align*}
		\left(1-\left|\frac{\lambda_{2}}{\lambda_{1}}(\r(h))\right|\right)^{-1}<\lambda_1^2 (\r(g)).
	\end{align*}
\end{thm}	
\begin{remark}
	A (slightly weaker) collar lemma for maximal representations $\rho:\G\to\Sp(2n,\R)$ was proven in \cite{BP}, however  \cite{BP} doesn't cover maximal representations in Hermitian Lie groups different from $\Sp(2n,\R)$, and our techniques of proof  are different. Our techniques of proof is also different from the technique of proof of the  collar lemma for Hitchin representations proven in \cite{Lee-Zhang}, albeit all the three proofs build on the topological input from hyperbolic geometry giving the orientation of a well chosen 6-tuple of fixed points of hyperbolic elements recalled in Figure~\ref{f.collar} below.
\end{remark}	
An important advantage of  the projective cross ratio $\cro_{\R\P^1}$ with respect to the cross ratios $\cro_k$, which is at the basis of an easy proof of the collar lemma for holonomies of hyperbolizations, is its additional symmetry:
$$\cro_{\R\P^1}(d,a,b,c)=\left(1- \cro_{\R\P^1}(a,b,c,d)^{-1}\right)^{-1}.$$
While this doesn't always hold for $\cro_1$, we establish the following generalization, which has the advantage of involving the exponential of the root (recall from Lemma~\ref{lem.cro-period} that  $\cro_1\left(h_+,hx,x,h_-\right)=\lambda_1(h)^2$).
\begin{lem}\label{p.collar2,n}
	Let $h\in\PO(2,q)$ be such that $|\lambda_1(h)|>|\lambda_2(h)|$. Denote by $h_\pm\in\Is_1(\R^{2,q})$ the eigenlines corresponding to the eigenvalues of highest and lowest absolute value. Then for every $x\in\Is_1(\R^{2,q})$  such that $(h_+,hx,x,h_-)$ is positive, 
	$$\left(1-\left|\frac{\lambda_2}{\lambda_1}(h)\right|\right)^{-1}< \cro_1\left(h_-,h_+,hx,x\right).$$
\end{lem}

\begin{proof}
	Since the 4-tuple $(h_+,hx,x,h_-)$ is positive, in particular the triple $(h_+,x,h_-)$ is positive. Thus we can assume without loss of generality that, with respect to the standard basis, $h_+=e_{q+2}$, $h_-=e_{1}$  and we can write, as in Example~\ref{e.3.9}, 
	$$x=[q_J(\sfs_x):\sfs_x:1]$$
	for some vector $\sfs_x\in c_J(V_J)$.
	
	{{We choose  the lift of $h$ to $\OO(2,q)$, which we  also denote by $h$, such that $\lambda_1:=\lambda_1(h)>0$.}}
	Since $h_+=e_{q+2}$ and $h_-=e_{1}$, the element $h$ acts on the subspace $V_J=\langle e_1,e_{q+2}\rangle^\perp$. We denote by $h_0:V_J\to V_J$ the induced linear map, {{which preserves the form $b_J$}}. Then we have
	$$\begin{array}{rl}
		hx&=[\lambda_1^{-1} q_J(\sfs_x): h_0 \sfs_x:\lambda_1]\\
		&=[\lambda_1^{-2} q_J(\sfs_x):\lambda_1^{-1} h_0 \sfs_x:1]
	\end{array}.$$
	Since the quadruple $(e_{q+2},hx,x,e_{1})$ is positive, in particular the triple $(e_{q+2},hx,e_{1})$ is positive, and thus $\lambda_1^{-1} h_0 \sfs_x\in c_J(V_J)$.
	
	Recall that $\tilde{x}\in \R^{2,q}$ denotes a non-trivial lift of $x\in\Is_1(\R^{2,q})$. We can now explicitly compute\footnote{Recall from Equation \eqref{eq.the from Q} that $Q((a,\sfs,b)^t,(c,\mathsf{t},d)^t)=-ad-bc+2b_J(\sfs,\mathsf{t})$.}
	\begin{align}
		\cro_1\left(h_-,h_+,hx,x\right)&=\displaystyle{\frac{Q(\tilde{h}_-,\tilde{hx})Q(\tilde{h}_+,\tilde x)}{Q(\tilde{h}_-,\tilde{h}_+)Q(\tilde{x},\tilde{hx})}}\nonumber\\ 
		&=\displaystyle{\frac{\lambda_1^2 q_J(\sfs_x)}{(\lambda_1^2+1)q_J(\sfs_x)-2b_J(\sfs_x,\lambda_1 h_0 \sfs_x)}}.\label{e.cr1}
	\end{align} 
	In order to conclude we need to find a good lower bound on the positive number $2b_J(\sfs_x,\lambda_1 h_0 \sfs_x)$, which is positive by Lemma~\ref{lem.pre.positiv bJ}. Setting $v_x:=\frac{1}{\sqrt{q_J(\sfs_x)}}\sfs_x$, we have that $v_x$ and $h_0 v_x$ belong to $q_J^{-1}(1)$. Therefore $v_x$ and $h_0 v_x$ are in the $-1$ level set of the quadratic from $-q_J$ of signature $(q,1)$. Since, by assumption, $v_x$ and $h_0v_x$ both have positive first entry with respect to the basis $(e_2,\ldots,e_{q+1})$, they are also in the same connected component, denoted by $S^+$, of $-q_J^{-1}(-1)$ , which is preserved by  $h_0$. {This also yields that the eigenvalue of $h_0$ of greatest modulus, which we denote by $\lambda_2$ and which coincides with $\lambda_2(h)$, is positive. }
	
	Observe that $S^+$ is the hyperboloid model of real hyperbolic $(q-1)$-space $\mathbb{H}^{q-1}$. Therefore
	$$\mathrm{arccosh} (b_J(v_x,h_0 v_x))=d_{\mathbb{H}^q}(v_x,h_0 v_x)\geq \ell_{\mathbb{H}^q} (h_0)$$
	where $\ell_{\mathbb{H}^q}$ denotes the translation length of the hyperbolic isometry $h_0$. Basic hyperbolic geometry yields that 
	$$\cosh (\ell_{\mathbb{H}^q} (h_0))=\frac{1}{2}(\lambda_2+\lambda_2^{-1})$$
	and therefore
	$$2b_J(\sfs_x,\lambda_1 h_0 \sfs_x)=2q_J(\sfs_x)\lambda_1(v_x,h_0 v_x)\geq q_J(\sfs_x)\lambda_1(\lambda_2+\lambda_2^{-1}).$$
	As a result, using that $\lambda_2>0$ and thus also $(\lambda_2+\lambda_2^{-1})>0$, we obtain 	\[
	\begin{array}{rl}
		\cro_1\left(h_-,h_+,hx,x\right)&=\displaystyle{\frac{\lambda_1^2 q_J(\sfs_x)}{(\lambda_1^2+1)q_J(\sfs_x)-2b_J(\sfs_x,\lambda_1 h_0 \sfs_x)}}\\
		&\geq\displaystyle{\frac{\lambda_1^2 q_J(\sfs_x)}{(\lambda^2+1)q_J(\sfs_x) -\lambda_1(\lambda_2+\lambda_2^{-1})q_J(\sfs_x)}}\\
		&=\displaystyle{\frac{\lambda_1}{(\lambda_1-\lambda_2)(1-\lambda_1^{-1}\lambda_2^{-1})}}\\ 
		&> \displaystyle{\frac{\lambda_1}{\lambda_1-\lambda_2}} =\frac{1}{1-\frac{ \lambda_2}{\lambda_1}}=\frac{1}{1-|\frac{ \lambda_2}{\lambda_1}|}.
	\end{array} 
	\]
\end{proof}	
The proof of Theorem~\ref{t.collarmax} follows combining Lemma~\ref{p.collar2,n} and the positivity of the cross ratio (Theorem~\ref{thm.pos ratioed}). Given $x\in\dg$ we denote as usual by $x^1\in\Is_1(\R^{2,q})$ the image of the boundary map, and given $g\in\G$ we denote by $g_\rho\in\PO(2,q)$ the element $\rho(g)$. 
\begin{proof}[Proof of Theorem~\ref{t.collarmax}]
		\begin{figure}[h]
			\begin{tikzpicture}[scale=1.2]
				\draw (0,0) circle [radius =1];
				\draw (-1,0) [blue, thick] to (1,0);
				\node at (-1,0) [left] {$g_+$};
				\node at (1,0) [right] {$g_-$};
				\draw (0,-1) [red, thick] to (0,1);
				\node at (0,1) [above] {$h_+$};
				\node at (0,-1) [below] {$h_-$};
				
				\node at (-0.5,0.9) [above] {$hg_+$};
				\filldraw (-0.36,0.92) circle [radius=1pt];
				\draw (0.36,0.92) [blue, thick] to (-0.36,0.92);
				\filldraw (-0.92,0.36) circle [radius=1pt] node [above left] {$gh_+$};
				\draw (-0.92,0.36) [red, thick] to (-0.92,-0.36);
				\filldraw (1,0) circle [radius=1pt];
				
				\filldraw (-1,0) circle [radius=1pt];
				\filldraw (0,1) circle [radius=1pt];
				\filldraw (0,-1) circle [radius=1pt];
				
			\end{tikzpicture}
			\caption{The proof of Theorem~\ref{t.collarmax}.}\label{f.collarm}
		\end{figure}
	We know  from Lemma~\ref{p.collar2,n} that 
	$$\left(1-\left|\frac{\lambda_2}{\lambda_1}(h)\right|\right)^{-1}< \cro_1\left(h_-^1,h_+^1,h_\rho g_+^1,g_+^1\right).$$
	Since the points $(h_-, g_-, h_+, hg_+, gh_+, g_+)\in\dg^6$ are cyclically ordered (compare Figure~\ref{f.collarm}, and \cite[Lemma 2.2]{Lee-Zhang}), and $\cro_1$ is positive along the image of the boundary map, we deduce using the cocycle identity (Proposition~\ref{prop.property of grassmannian cro} (3-4))
		\begin{align*}
			\cro_{1}(h^{1}_-,h^{1}_+,h_\rho g^{1}_+,g_+^{1})&<\cro_{1}(g^{1}_-,h^{1}_+,h_\rho g^{1}_+,h_-^{1})\cro_{1}(h^{1}_-,h^{1}_+,h_\rho g^{1}_+,g_+^{1})\\
	&=\cro_{1}(g^{1}_-,h^{1}_+,h_\rho g^{1}_+,g_+^{1})\\
	&<\cro_{1}(g^{1}_-,h^{1}_+,h_\rho g^{1}_+,g_+^{1})\cro_{1}(g^{1}_-,h_\rho g^{1}_+,g_\rho h^{1}_+,g_+^{1})\\
	&=\cro_{1}(g^{1}_-,h^{1}_+,g_\rho h^{1}_+,g_+^{1}).
		\end{align*}
		The theorem follows from Lemma~\ref{lem.cro-period}, stating
		\[\cro_{1}(g^{1}_-,h^{1}_+,g_\rho h^{1}_+,g_+^{1})=\lambda_1^2 (g_{\r}). \qedhere\]
\end{proof}	
\subsection{Hybrid flags and strategy of proof in the general case}\label{s.c2}
For the proof of the collar lemma we will need the following construction.
\begin{defn}\label{def.hybrid flag}
	Given two transverse flags  $x,y\in \calF_{p-1}$, the \emph{$(x,k)$-hybrid flag} is
	$$ x\vartriangleleft_k y:= (x^1,\ldots, x^{k-1}, x^{k-1}\oplus (x^{k+1}\cap y^{p+q-k}),x^{k+1}, \ldots, x^{p-1})$$
	if $k=1,\ldots,p-2$ and  
	$$ x\vartriangleleft_{p-1} y:= (x^1,\ldots, x^{p-2}, x^{p-2}\oplus (x^{q+2}\cap y^{p-1}))$$
	otherwise.
\end{defn}
In the case of the standard flags $\sX,\sZ$ where  $\sZ^k=\<e_1,\ldots,e_{k}\>$ the hybrid flags are given, 
$$\sZ  \vartriangleleft_k \sX=\left\{\begin{array}{ll}
	(\sZ^1,\ldots,\sZ^{k-1},\<e_1,\ldots,e_{k-1},e_{k+1}\>, \sZ^{k+1},\ldots, \sZ^{p-1} ) &\text{ if } k<p-1\\
(\sZ^1,\ldots,\sZ^{p-2}, \<e_1,\ldots,e_{p-2},e_{q+2}\> )&\text{ if } k=p-1.
	\end{array}
	\right.$$

The goal of the section is to prove the collar lemmas; more specifically we want to show that for  any $\T$-positive representation $\r$, any linked pair $g,h\in\G\backslash\{e\}$, and any  $k<p-1$, 
$$		\left(1-\left|\frac{\lambda_{k+1}}{\lambda_{k}}(h_\r)\right|\right)^{-1}<\lambda_1^2 \cdots\lambda_k^2 (g_\r).$$

We will prove the collar lemma  in two steps: 
First we show that 
\begin{align}\label{e.STEP1}
	\left(1-\left| \frac{\lambda_{k+1}}{\lambda_k}(h_{\r})\right|\right)^{-1}<\cro_k\left(h^k_-,h^k_+,gh^k_+,(h_-\vartriangleleft_k g_+)^k\right);
\end{align}
in the case $k=p-1$ we show how to reduce this claim to Lemma~\ref{p.collar2,n} in Section~\ref{S.Step1}, building on the study of some level sets of the cross ratio $\cro_k$ carried out in Proposition~\ref{prop.projection of cross ratio}.
If $k<p-1$  Equation \eqref{e.STEP1} is simpler and directly follows from property $H_{k}$ \cite{Beyrer-Pozzetti}. 

In the second step we prove that $\cro_k(h^k_-,h^k_+,gh^k_+,g^k_+)$ is bigger than the right hand side. The main ingredients used in this step are a technical statement (Proposition~\ref{prop.k-projection is positive}) about   $k$--hybrid flags, which we prove in Section~\ref{s.hybrid}, and the connection between positivity of triples and positivity of the cross ratio (Proposition~\ref{prop.derivative cross ratio for positive triple: general}). We conclude using that the representations are $k$--positively ratioed and Lemma~\ref{lem.cro-period}.

\begin{remark}
	Note that since the Jordan projections of $g$ and $g^{-1}$ agree,  it is enough to prove the collar lemma only for one of the pairs $(g,h)$ and $(g^{-1},h)$. As a result we can and will assume, possibly replacing $g$ with $g^{-1}$, that the quadruple $(h_-,h_+,gh_+,g_+)$ is positive: 
\end{remark}

\subsection{Step 1}\label{S.Step1}
The first step for $k<p-1$ follows  directly from \cite{Beyrer-Pozzetti}.

\begin{prop}[{\cite[Corollary 6.3]{Beyrer-Pozzetti}}]\label{prop.step 1 small k}
	Let $\r:\G\to\PO(p,q)$ satisfy property $H_{k}$, then for $k<p-1$ and any linked pair $g,h\in\G\backslash\{e\}$ we have
	$$\left(1-\left| \frac{\lambda_{k+1}}{\lambda_{k}}(h_{\r})\right|\right)^{-1}<\cro_{k}\left(h^{k}_-,h^{k}_+,gh^{k}_+,(h_-\vartriangleleft_k g_+)^k\right)$$ 
\end{prop}

Thus we are left to consider the case $k=p-1$.

\begin{prop}\label{prop.step 1 k=p-1}
	Assume that $\r:\G\to\PO(p,q)$ is $\{p-2,p-1\}$--Anosov and the projections $\xi_x$ to $\quotient{x^{q+2}}{x^{p-2}}$ are $\T$-positive curves for all $x\in \dg$. Then for any linked pair $g,h\in\G\backslash\{e\}$ we have
	$$\left(1-\left|\frac{\lambda_{p}}{\lambda_{p-1}}(h_{\r})\right|\right)^{-1}<
	\cro_{p-1}\left(h^{p-1}_-,h^{p-1}_+,gh^{p-1}_+,(h_-\vartriangleleft_{p-1} g_+)^{p-1}\right).$$ 
\end{prop}

\begin{proof}
		We consider the vector space $H=\quotient{h^{q+2}_-}{h^{p-2}_-}$ with the induced non-degenerate from of signature $(2,q-p+2)$. Of course $h_\rho$ induces an element $\ov h\in\PO(2,q-p+2)=\PO(H)$
	satisfying
	$\frac{\lambda_{p-1}}{\lambda_p}(h_{\r})=  \frac{\lambda_{1}}{\lambda_2}(\ov h)$.
	
	We consider the $\T$-positive curve $\xi_{h_-}:\dg\to\calF_1(\quotient{h^{q+2}_-}{h^{p-2}_-})$ defined as in Corollary~\ref{cor.projections are positive}: $\xi_{h_-}(x)=[(x^{p-1}\cap h_-^{q+2})\oplus h_-^{p-2}]$, $\xi_{h_-}(h_-)=[h_-^{p-1}]$. 	 In order to make the notation lighter we will denote by $[x^{p-1}]\in\Is_1(H)$ the image $\xi_{h_-}(x)$.
	
	Since $\xi_{h_-}$ is $h_\rho$ equivariant and positive (Corollary~\ref{cor.projections are positive}), we deduce from Lemma ~\ref{p.collar2,n} that 
	\begin{equation}\left(1-\left|\frac{\lambda_{p}}{\lambda_{p-1}}(h_{\r})\right|\right)^{-1}<\cro_1^H([h^{p-1}_-],[h^{p-1}_+],[hg^{p-1}_+],[g^{p-1}_+])
	.\end{equation}

	Since $\xi_{h_-}$ is $\T$-positive   and the 5-tuple $(h_-,h_+,hg_+, gh_+,g_+)\in\dg^{(5)}$ is positively oriented (see \cite[Lemma 2.2]{Lee-Zhang}, and Figure~\ref{f.collar} below), we can apply Proposition~\ref{p.poscr1} to derive that $$\cro_1^H([h^{p-1}_-],[hg^{p-1}_+],[gh^{p-1}_+],[g^{p-1}_+])>1.$$ Since  the cocycle identity gives that $\cro_1^H([h^{p-1}_-],[h^{p-1}_+],[gh^{p-1}_+],[g^{p-1}_+])$ equals
	$$\cro_1^H([h^{p-1}_-],[h^{p-1}_+],[hg^{p-1}_+],[g^{p-1}_+])\cro_1^H([h^{p-1}_-],[hg^{p-1}_+],[gh^{p-1}_+],[g^{p-1}_+]),$$  we deduce
\begin{equation}\cro_1^H\mkern-2mu([h^{p-1}_-\mkern-2mu],[h^{p-1}_+\mkern-2mu],[hg^{p-1}_+\mkern-2mu],[g^{p-1}_+\mkern-2mu])<\cro_1^H\mkern-2mu([h^{p-1}_-\mkern-2mu],[h^{p-1}_+\mkern-2mu],[gh^{p-1}_+\mkern-2mu],[g^{p-1}_+\mkern-2mu])
.\end{equation}

We can then use Proposition~\ref{prop.projection of cross ratio} to compute the right hand side as
	$$\cro_1^H([h^{p-1}_-\mkern-2mu],[h^{p-1}_+\mkern-2mu],[gh^{p-1}_+\mkern-2mu],[g^{p-1}_+\mkern-2mu])=\cro_{p-1}(h^{p-1}_-,h^{p-1}_+,gh^{p-1}_+,h_-^{p-2}\oplus (h_-^{q+2}\cap g^{p-1}_+)).$$
This concludes the proof once we observe that, by definition, 
$$(h_-\vartriangleleft_{p-1} g_+)^{p-1}= h_-^{p-2}\oplus \left(g_+^{p-1}\cap h_-^{q+2}\right).$$

\end{proof}

\subsection{Positivity of hybrid flags}\label{s.hybrid}
Recall from Definition~\ref{def.hybrid flag} that to a pair of flags $x,y\in \calF_{p-1}$ and an integer $1\leq k\leq p-1$ we associate the $(x,k)$-hybrid flag $x\vartriangleleft_k y$. The goal of the section is to prove that positivity is preserved under hybridization. { We will need the following basic lemma

\begin{lem}\label{lem.5.8}
Let $(w,x,y,z)\in \calF_1^4$ be $\T$-positive. Then there is a continuous path $\zeta:[0,1]\to \calF_{1}$ from $x$ to $w$ such that $\zeta(t)$ is transverse to $y,z$ for all $t\in [0,1]$.
\end{lem}

\begin{proof}
Recall the notation from Example~\ref{e.3.9}. Up to the action by an element in $\PO(2,q)$, we can assume that $x=\sX$, $z=\sZ$ $y=P(\sfs_y) \sX$ and $w=P(-\sfs_w)\sX$ for $\sfs_y,\sfs_w\in c_J(V_J)$. Then  $\zeta(t):= P(-t \sfs_w)\sX$ is a continuous path from $x$ to $w$. Moreover $\zeta(t)$ is transverse to $y,z$ for all $t\in [0,1]$, since $(P(-t \sfs_w)\sX, P(\sfs_y) \sX,\sZ)$ is a positive triple: the image of this triple under $P(t \sfs_w)$ is $(\sX, P(t\sfs_w+\sfs_y) \sX,\sZ)$ .
\end{proof}

\begin{prop}\label{prop.k-projection is positive} Let $(w,x,y,z)$ be a $\T$-positive quadruple in $\calF^4_{p-1}$. For all $k=1,\ldots,p-1$ the triple 
	$(x\vartriangleleft_k w,y, z)\in\calF_{p-1}^3$
	is positive.
\end{prop}

\begin{proof}
Since $(x,y,z)\in\calF_{p-1}^3$ is positive, by Corollary~\ref{lem.deform positive tuples} it is enough to construct a continuous path $\eta:[0,1]\to \calF_{p-1}$  with $\eta(0)=x$ and $\eta(1)=x\vartriangleleft_k w$ such that $\eta(t)$ is transverse to $y,z$ for all $t\in [0,1]$.

We consider the projection $\pi^k_x:\Is_{k}^{\tv x}\to\mathbb P\left({\quotient{x^{k+1}}{x^{k-1}}}\right)$ for $k< p-1$ (respectively $\pi^{p-1}_x:\Is_{p-1}^{\tv x}\to\calF_1\left({\quotient{x^{q+2}}{x^{p-2}}}\right)$) defined in Section~\ref{s.positive-bdy-map}. 
Choose a continuous path $\zeta_{p-1}:[0,1]\to\calF_{1}(\quotient{x^{q+2}}{x^{p-2}})$ from $[x^{p-1}]$ to $\pi^{p-1}_x(w)$ as in Lemma~\ref{lem.5.8} for the quadruple $(\pi^{p-1}_x(w),[x^{p-1}],\pi^{p-1}_x(y),\pi^{p-1}_x(z))$. The latter quadruple is positive by Proposition~\ref{prop.projections are positive}.
For $k<p-1$ choose a continuous path $\zeta_{k}:[0,1]\to\mathbb P(\quotient{x^{k+1}}{x^{k-1}})$ from $[x^k]$ to $\pi_x^k(w)$ not passing through $\pi_x^k(y)$ and $\pi_x^k(z)$. This is possible according to Proposition~\ref{lem.kth-projection}.

We set $\eta(0):=x$. For $t>0$, if $k<p-1$, we set $\eta(t)^{i}:=x^{i}$ for $i\neq k$ and $\eta^k(t)$ to be the unique point of $\Is_k$ that contains $x^{k-1}$, is contained in $x^{k+1}$ and satisfies $\pi_x^{k}(\eta(t)^k)=\zeta_k(t)$ in $\P(\quotient{x^{k+1}}{x^{k-1}})$. Analogously, for $k=p-1$, we set $\eta(t)^{i}:=x^{i}$ for $i\neq p-1$ and $\eta^{p-1}(t)$ to be the unique point of $\Is_{p-1}$ that contains $x^{p-2}$, is contained in $x^{q+2}$ and satisfies $\pi_x^{k}(\eta(t)^k)=\zeta_k(t)$ in $\Is_1(\quotient{x^{q+2}}{x^{p-2}})$.

By construction we have that $\eta$ is a continuous path from $x$ to $x\vartriangleleft_k w$. Since $\eta(t)^{i}=x^{i}$ for $i\neq k$ and $x$ is transverse to $y$ and $z$ it is enough to observe that $\eta(t)^k$ is transverse to $y^k$ and $z^k$, which is equivalent to $\pi^k_x(\eta(t))$ is transverse to $\pi^k_x(y^k)$ and $\pi^k_x(z^k)$. The latter holds by construction.
\end{proof}
}

\subsection{Step 2}

To finish the proof of the collar lemma we need the following fact, which is essentially a combination of Propositions~\ref{prop.k-projection is positive} and~\ref{prop.derivative cross ratio for positive triple: general}.

\begin{prop}\label{prop.last step before collar}
	Let $\r$ be $\T$-positive and $(w,x,y,z)$ be a $\T$-positive quadruple in $\dg^4$. Then for any $k\leq p-1$
	\begin{align*}
		\gcr_{k}\left(x^k,y^{k},z^{k},(x\vartriangleleft_k w)^k\right)\leq \gcr_{k}\left(x^k,y^{k},z^{k},w^k\right)
	\end{align*}
\end{prop}

\begin{proof}
First observe that Proposition~\ref{prop.k-projection is positive} applied to  the $\T$-positive quadruple $(\xi(w),\xi(x),\xi(y),\xi(z))$ guarantees that $(x\vartriangleleft_k w)^k$ is transverse to $y^k$ and $z^k$. Since also $x^k$  is transverse to  $y^k$ and $z^k$, the cross ratio on the left hand side of the statement is defined.
	Using the cocycle identity of the cross ratio the claim reduces to showing that 
	$$\gcr_{k}\left((x\vartriangleleft_k w)^k,y^{k},z^{k},w^k\right) \geq 1.$$
	We consider $\gcr_{k}\left((x\vartriangleleft_k w)^k,y^{k},y_t^{k},w^k\right)$ as a function of the point $y_t$ varying between $y$ and $z$. 
	Since $\xi^k$ has Lipschitz image, by Rademacher's theorem and the fundamental theorem of calculus we know that $\gcr_{k}\left((x\vartriangleleft_k w)^k,y^{k},z^{k},w^k\right)$ is the integral of the derivative of this function - cfr. the proof of Theorem~\ref{thm.pos ratioed}. As a result, using again the cocycle identity, it is enough to show that, 	for every $y_0\in (y,z)_x$ such that $\dot \xi^k(y_0)$ is defined, and for some $C^1$--approximation  $t\to y_t^k$ of $\xi^k$ at $y_0^k$, it holds 
	$$\left.\frac{d}{dt}\right|_{t=0} \gcr_{k}\left((x\vartriangleleft_k w)^k,y_0^{k},y_t^{k},w^k\right) \geq 0;$$
here once again  Proposition~\ref{prop.k-projection is positive} applied to the 4-tuple $(\xi(w),\xi(x),\xi(y_0),\xi(z))$) ensures that the cross ratio is defined.
	
	\begin{figure}[h]
		\begin{tikzpicture}[scale=.8]
			\draw (0,0) circle [radius =1];
			\node at (1,0) [right] {$w$};
			\node at (-0.5,0.9) [above] {$y_0$};
			\node at (-1,0) [left] {$y$};
			\node at (0,1) [above] {$y_t$};
			\node at (0,-1) [below] {$x$};
			\node at (.8,.8)[right]{$z$};
			\node at (.8,-.8)[right]{$u$};
			\filldraw (.7,.7) circle [radius=1pt];
			\filldraw (1,0) circle [radius=1pt];
			\filldraw (-0.36,0.92) circle [radius=1pt];
			\filldraw (-1,0) circle [radius=1pt];
			\filldraw (0,1) circle [radius=1pt];
			\filldraw (0,-1) circle [radius=1pt];
			\filldraw (.7,-.7) circle [radius=1pt];
			\node at (0,0.3) [below] {$\circlearrowright$};
		\end{tikzpicture}
		\caption{The order of $x,y,y_0,y_t,z,w,u$ for $t>0$.}
	\end{figure}

	Proposition~\ref{prop.k-projection is positive} together with Proposition~\ref{c.permutation} implies that for all $u\in (w,x)_y$ the triple $(\xi(y_0),\xi(w),\xi(x)\vartriangleleft_k \xi(u))$ is $\T$-positive. Thus we can apply Proposition~\ref{prop.derivative cross ratio for positive triple: general} to derive that
	$$\left.\frac{d}{dt}\right|_{t=0} \gcr_{k}\left((x\vartriangleleft_k u)^k,y_0^{k},y_t^{k},w^k\right) > 0.$$
	Now the regularity of the cross ratio yields that
	$$\left.\frac{d}{dt}\right|_{t=0} \gcr_{k}\left((x\vartriangleleft_k u)^k,y_0^{k},y_t^{k},w^k\right) \to \left.\frac{d}{dt}\right|_{t=0} \gcr_{k}\left((x\vartriangleleft_k w)^k,y_0^{k},y_t^{k},w^k\right)$$
	for $u\to w$. This proves the claim.
\end{proof}

The collar lemmas can now be proved following the same lines as the proof of Theorem~\ref{t.collarmax}:

\begin{thm}\label{t.collar}
	Let $\r:\G\to\PO(p,q)$ be a $\T$-positive Anosov representation and $g,h\in\G$ a linked pair. Then for any $k\leq p-1$
	\begin{align*}
		\left(1-\left|\frac{\lambda_{k+1}}{\lambda_{k}}(h_{\r})\right|\right)^{-1}<\lambda_1^2 \cdots\lambda_k^2 (g_{\r}).
	\end{align*}
\end{thm}

\begin{proof}
	\begin{figure}[h]
		\begin{tikzpicture}[scale=1.2]
			\draw (0,0) circle [radius =1];
			\draw (-1,0) [blue, thick] to (1,0);
			\node at (-1,0) [left] {$g_+$};
			\node at (1,0) [right] {$g_-$};
			\draw (0,-1) [red, thick] to (0,1);
			\node at (0,1) [above] {$h_+$};
			\node at (0,-1) [below] {$h_-$};
			
			\node at (-0.5,0.9) [above] {$hg_+$};
			\filldraw (-0.36,0.92) circle [radius=1pt];
			\draw (0.36,0.92) [blue, thick] to (-0.36,0.92);
			\filldraw (-0.92,0.36) circle [radius=1pt] node [above left] {$gh_+$};
			\draw (-0.92,0.36) [red, thick] to (-0.92,-0.36);
			\filldraw (1,0) circle [radius=1pt];
			
			\filldraw (-1,0) circle [radius=1pt];
			\filldraw (0,1) circle [radius=1pt];
			\filldraw (0,-1) circle [radius=1pt];
			
			\draw (-1,0) to (0,-1);
			\filldraw (-0.6,-0.4) circle [radius=1pt] node [below left]{$h_-\vartriangleleft_k g_+$};
		\end{tikzpicture}
		\caption{The proof of Theorem~\ref{t.collar}.}\label{f.collar}
	\end{figure}
	From Propositions~\ref{prop.step 1 small k} and~\ref{prop.step 1 k=p-1} we know that for any $k\leq p-1$
	\begin{align*}
		\left(1-\left|\frac{\lambda_{k+1}}{\lambda_{k}}(h_{\r})\right|\right)^{-1}<\cro_{k}\left(h^{k}_-,h^{k}_+,gh^{k}_+,(h_-\vartriangleleft_k g_+)^{k}\right).
	\end{align*}
	Proposition~\ref{prop.last step before collar} implies
	\begin{align*}
		\cro_{k}(h^{k}_-,h^{k}_+,gh^{k}_+,(h_-\vartriangleleft_k g_+)^{k})\leq\cro_{k}(h^{k}_-,h^{k}_+,gh^{k}_+,g_+^{k}).
	\end{align*}
	Since $\r$ is $k$--positively ratioed and $h_-,g_-,h_+,gh_+,g_+$ are in that order on $\dg$ (which implies $\cro_{k}(g^{k}_-,h^{k}_+,gh^{k}_+,h_-^{k})>1$), we derive with the cocycle identity that 
	\begin{align*}
		\cro_{k}(h^{k}_-,h^{k}_+,gh^{k}_+,g_+^{k})<\cro_{k}(g^{k}_-,h^{k}_+,gh^{k}_+,g_+^{k}).
	\end{align*}
	The theorem follows from Lemma~\ref{lem.cro-period}, stating
\[\cro_{k}(g^{k}_-,h^{k}_+,gh^{k}_+,g_+^{k})=\lambda_1^2 \cdots\lambda_k^2 (g_{\r}). \qedhere\]
\end{proof}

\section{Higher rank Teichm\"uller theory for $\PO(p,q)$}

In this section we show 
that $\T$-positive Anosov representations into $\PO(p,q)$ form higher rank Teichm\"uller spaces. In other words we show that being $\T$-positive Anosov is a closed condition in the character variety 
$$\Xi(\PO(p,q))=\quotient{\Homr(\G,\PO(p,q))}{\PO(p,q)}.$$
A direct consequence of Corollary~\ref{lem.deform positive tuples} is that the set of $\T$-positive representations is open within the set of Anosov representations, and thus the set of $\Theta$-positive Anosov representations is open in the character variety. We then derive that $\Theta$-positive Anosov representations form connected components of $\Xi( \PO(p,q))$.

We begin with an easy consequence of Theorem~\ref{t.ccomp} due to Guichard-Wienhard \cite{GWPaper}.

\begin{prop}\label{cor.positive clopen under Anosov}
	$\T$-positive Anosov representations form a union of connected components of the set $\Hom_{\T}\subset  \Hom(\G,\PO(p,q))$  of $\T$-Anosov representations. 
\end{prop}

\begin{proof}
	Since open connected subsets of a real algebraic variety are path connected, 
	 it is enough to verify that if $\r_{*}:[0,T]\to \Hom_{\T}$ is a continuous path, and $\r_0:\G\to \PO(p,q)$ is $\T$-positive, then $\rho_T$ is $\T$-positive. We denote by $\xi_t:\dg\to \calF_{p-1}$ the boundary map associated to the representation $\rho_t$, which is transverse because $\rho_t$ is Anosov. It is well known that  $\xi_t$ depends continuously on $\rho_t$ \cite[Theorem 5.13]{Guichard-Wienhard-IM}, as a result, given a $\T$-positive $n$--tuple $(x_1,\ldots,x_n)\in \dg^{(n)}$, the path $(\xi_t(x_1),\ldots,\xi_t(x_n))$ is a continuous transverse path. Corollary~\ref{lem.deform positive tuples} applies, and gives that also $(\xi_T(x_1),\ldots,\xi_T(x_n))$ is a $\T$-positive $n$--tuple.
\end{proof}

A \emph{semi-simplification} of a representation $\r^{ss}:\G\to\PO(p,q)$ of $\r:\G\to\PO(p,q)$ can be characterized as a representation in the unique closed orbit inside the closure $\ov{\PO(p,q)\cdot \r} \subset\Hom(\G,\PO(p,q))$, where the action of $\PO(p,q)$ is by conjugation (cfr. \cite[Proposition 2.39]{GGKW}). Thanks to the semi-simplification, it is enough  to consider  reductive limits of representations, instead of general limits: if $\{\r_n\}_{n\in \N}$ converges to $\r_0$ in $\Hom(\G,\PO(p,q))$, then we can find conjugates $\{g_n \r_n g_n^{-1}\}_{n\in \N}$ converging to $\r^{ss}_0$. Moreover any semi-simplification is reductive.

\begin{cor}\label{cor.positive for semisimplification}
	A representation $\r:\G\to\PO(p,q)$ is $\T$-positive Anosov if and only if its semi-simplifications $\r^{ss}$ are.
\end{cor}

\begin{proof}
	It is known that $\r$ is $\T$-Anosov if and only if $\r^{ss}$ is (\cite[Proposition 2.39]{GGKW}). Since $\r$ and $\r^{ss}$ belong to the same connected component of $\Hom_{\T}$, the result follows from Proposition~\ref{cor.positive clopen under Anosov}.
\end{proof}

{Since $\T$-Anosov representations form open sets of $ \Hom(\G,\PO(p,q))$, it follows again from Proposition~\ref{cor.positive clopen under Anosov} that also $\T$-positive Anosov representations, a union of connected components of the former set, form open sets of $ \Hom(\G,\PO(p,q))$.} It remains to show that  $\T$-positive Anosov representations also form closed subsets of $ \Hom(\G,\PO(p,q))$. For this we crucially need the following result from \cite{Beyrer-Pozzetti2}: any proximal and reductive limit of positively ratioed representations admits a well behaved bounday map.

\begin{thm}[{\cite[Theorem B]{Beyrer-Pozzetti2}}]\label{thm.pos ratio limit}
	Let $\{\r_n:\G\to \PO(p,q)\}_{n\in \N}$ be a sequence of $k$--positively ratioed representations converging to a $k$--proximal reductive representation $\r_0:\G\to \PO(p,q)$. Then $\r_0$ admits an equivariant transverse continuous boundary map
	$\xi^{k}_{\r_0}:\dg\to \Is_k(\R^{p,q}),$
	which is dynamics preserving\footnote{Recall Definition~\ref{d.Anosov} (ii).} at $\g_+\in\dg$ for any $\g\in\G$ with $\r_0(\g)$  $k$--proximal.
\end{thm}

An element $g\in\PO(p,q)$ is \emph{$k$-proximal} if it has a unique attracting point in $\Is_k(\R^{p,q})$, equivalently if $|\frac{\lambda_k}{\lambda_{k+1}}(g)|>1$;  a representation is  \emph{$k$--proximal} if its image contains a $k$--proximal element. {
	We prove closedness in two steps. First we deduce from Theorem~\ref{thm.intro-positive}, Theorem~\ref{thm.intro-collar} and Theorem~\ref{thm.pos ratio limit}  that any limiting representation admits well behaved boundary maps:
\begin{prop}\label{p.closdness of positivity}
	Let $\{\r_n:\G\to \PO(p,q)\}_{n\in \N}$ be $\T$-positive Anosov representations  converging to reductive representation $\r_0$ in $\Hom(\G,\PO(p,q))$. Then $\r_0$ admits equivariant, dynamics preserving, $\Theta$-positive maps.
\end{prop}

\begin{proof} 

	From the collar lemma (Theorem~\ref{thm.intro-collar}) it follows that $\rho_0(\g)$ is $k$--proximal for every $k=1,\ldots,p-1$ and every $\g\in\G\backslash\{e\}$. Indeed, we can find a linked element $g\in\G\backslash\{e\}$ and thus  Theorem~\ref{t.collar}  gives
	\begin{align*}
		\left(1-\left| \frac{\lambda_{k+1}}{\lambda_{k}}(\r_n(\g))\right| \right)^{-1}<\lambda_1^2 \cdots\lambda_k^2 ({\r_n}(g)).
	\end{align*}
	As both quantities converge to the corresponding quantities for the representation $\rho_0$, we deduce that 
	$\left|\frac{\lambda_{k}}{\lambda_{k+1}}({\r_0}(\g))\right|>1$.

	Since every $\r_n$ is $k$--positively ratioed for $k=1,\ldots,p-1$ (Theorem~\ref{thm.intro-positive}) we can apply Theorem~\ref{thm.pos ratio limit} and deduce that $\r_0$ admits an equivariant, transverse, continuous, dynamics preserving boundary map
	$$\xi_{\r_0}:\dg\to \calF_{p-1}.$$
	It follows from Corollary~\ref{lem.deform positive tuples} that such boundary map is furthermore $\Theta$-positive. Indeed since $\xi_{\r_0}$ is dynamics preserving, we have $\xi_{\r_n}(\g_+)\to \xi_{\r_0}(\g_+)$
	for every attracting fixed point $\g_+\in\dg$ of $\g\in \G\backslash\{e\}$. The same argument as in the proof of Proposition~\ref{cor.positive clopen under Anosov} gives that the restriction of $\xi_{\rho_0}$ to the set of fixed points of elements in $\G\backslash\{e\}$ is a $\T$-positive map, and thus Corollary~\ref{cor.pos tuples} implies that $\xi_{\rho_0}$ is $\T$-positive.
\end{proof}
Second, we show that this is enough to guarantee that the limiting representation is Anosov:  for this we introduce  a new way to use the $\T$-positive boundary map to read eigenvalue gaps for the representation $\rho_0$.
\begin{prop}\label{p.positiveAnosov}
A representation $\rho$ admitting equivariant, continuous, $\Theta$-positive, dynamics preserving boundary maps $\xi_{\r}$ is $\Theta$-positive Anosov. 	
\end{prop}	
\begin{proof}	
Since $\Theta$-positve boundary maps are by definition transverse,	it remains to show the third condition in our definition of Anosov representations, namely that for a sequence $\g_n\in \G$ with $|\g_n|_{\infty}\to\infty$ it holds: 
	$$\left|\frac{\lambda_{k}}{\lambda_{k+1}}(\r_0(\g_n))\right|\to\infty.$$ 
	Up to conjugating and extracting a subsequence we can assume that $\g_n^+\to l_+$, $\g_n^-\to l_-$, and $l_-\neq l_+$. Pick $x\in \dg$ different from $l_-$ and $l_+$.

	{{
			It follows from the proofs of Proposition~\ref{prop.step 1 small k} and~\ref{prop.step 1 k=p-1} (cfr. Lemma~\ref{p.collar2,n}) that
			$$1-\left|\frac{\lambda_{k+1}}{\lambda_{k}}(\r(\g_n))\right|\geq \cro_{k}\left((\g_n^-)^{k},(\g_nx)^{k},(\g_n^+)^{k},(\g_n^-\vartriangleleft_{k}x)^{k}\right).$$
			It is enough to show that the right hand side converges to 1. If $k<p-1$ set $L:=\quotient{l_-^{k}}{l_-^{k-1}}$ and if $k=p-1$ set $L:=\quotient{l_-^{q+2}}{ l_-^{p-2}}$. The boundary dynamics of the surface group guarantees that $\g_n x$ converges to $l_+$; using continuity of the cross ratio we deduce 
			\begin{align*}
				\cro_{k}\left((\g_n^-)^{k},(\g_nx)^{k},(\g_n^+)^{k},(\g_n^-\vartriangleleft_{k}x)^{k}\right)
				\to \cro_{k}((l_-)^{k},(l_+)^{k},(l_+)^{k},(l_-\vartriangleleft_{k}x)^{k}).
			\end{align*}
			We apply Proposition~\ref{prop.projection of cross ratio} to the latter cross ratio. If $k= p-1$ the cross ratio in the limit equals 
			$$\cro_1^{L}([l^{p-1}_-],[l^{p-1}_+\cap l_-^{q+2}],[l^{p-1}_+\cap l_-^{q+2}],[x^{p-1}\cap l_-^{q+2}]).$$ Since $\xi_{\r}$ is a $\T$-positive map, by Proposition~\ref{prop.projections are positive} the last cross ratio is defined (i.e. $[l^{p-1}_+\cap l_-^{q+2}]\tv [l^{p-1}_-],[x^{p-1}\cap l_-^{q+2}]$), thus it is equal to 1, as desired. 
			
			If $k<p-1$ the cross ratio in the limit equals $$\cro_1^{L}([l^{k}_-],[l^{p+q-k}_+\cap l_-^{k+1}],[l^{p+q-k}_+\cap l_-^{k+1}],[x^{p+q-k}\cap l_-^{k+1}]).$$ 
			As $\xi_{\r}$ satisfies property $H_k$ (Corollary~\ref{prop.Hk for pos}), it follows that $$[l^{p+q-k}_+\cap l_-^{k+1}]\tv [x^{p+q-k}\cap l_-^{k+1}],[l^{k}_-].$$
			 Therefore the cross ratio is again defined and equal to 1. In particular $\r_0$ is $k$--Anosov for $k\leq p-1$, as desired.}}
	
\end{proof}
This concludes the proof of Theorem~\ref{thm.C-intro}, namely that the set of $\Theta$-positive Anosov representations is closed in the representation variety: it follows from Propositions \ref{p.closdness of positivity} and \ref{p.positiveAnosov}   that any reductive limit of $\Theta$-positive Anosov representations is $\Theta$-positive Anosov.  Corollary \ref{cor.positive for semisimplification} and the discussion just before the corollary yield that \emph{any} limit of $\Theta$-positive Anosov representations is $\Theta$-positive Anosov}.

{Theorem~\ref{thm.C-intro} has applications to character varieties, which, as in the introduction, we realize as the quotient of the set of reductive representations by the $\PO(p,q)$-action by conjugation. { It follows combining Corollary \ref{c.higherTM spaces} and the work of \cite{ABC-IM} that $\T$-positive Anosov representations in $\SO(p,q)$ are necessarily reductive, since it is proven in \cite{ABC-IM} that all representations in the components of the caracter variety corresponding to $\T$-positive Anosov representations are non-parabolic, while a non-reductive representation, as well as its semisimplification, is parabolic. 
}}


\begin{cor}\label{c.higherTM spaces}
The subset consisting (of conjugacy classes) of $\T$-positive Anosov representations is a union of connected components of 
	$$\Hom(\G,\PO(p,q))\quad\text{ and }\quad\Xi( \PO(p,q)).$$
\end{cor}


 Corollary~\ref{c.higherTM spaces} also shows that $\T$-positive Anosov representations constitute connected components in $\Xi(\SO(p,q))$ (cfr. Remark~\ref{rem.pos general groups}).

\begin{remark}\label{rem.ht-components}
	Aparicio-Arroyo--Bradlow--Collier--Garcia-Prada--Gothen--Oliveira used  Higgs bundle techniques to count  the connected components of the character variety $\Xi(\SO(p,q))$ \cite{ABC-IM} and checked for the existence of $\T$-positive representations.
	The component count in \cite{ABC-IM} is as follows, denoting by $g$ the genus of the Riemann surface:
	\begin{itemize}
		\item There are $2^{2g+2}$ \emph{mundane} components. Each of these compontents contains a representation with image contained in a compact subgroup of $\SO(p,q)$ - such a representation is clearly not discrete and thus not Anosov. It follows from Corollary~\ref{c.higherTM spaces} that those connected components contain no $\T$-positive Anosov representation.
		
		\item If $q>p+1$, then there are $2^{2g+1}$ additional \emph{exceptional} components. Each  such component contains a  $\T$-positive Anosov representation; namely a representation as in Example~\ref{e.positive fuchsian}.\footnote{ {{Actually, using the notation of Example~\ref{e.positive fuchsian}, a representation of the form $\r:\G\to\SO(p,q),\r:=\det(\alpha)(\tau\circ\iota)\oplus\alpha$.}}} Thus every representation in those connected components is $\T$-positive Anosov. The exceptional components are then, as conjectured, the higher rank Teichm\"uller spaces associated to the group $\SO(p,q)$ for $p+1<q$.
		
		\item If $q=p+1$ there are $2^{2g+1}-1+2p(g-1)$ further exceptional components, which are smooth and conjectured to consit {{entirely of \emph{Zariski dense} representations \cite[Conjecture 1.7]{Collier-ENS}. According to \cite[Remark 7.14]{ABC-IM},  Corollary~\ref {c.higherTM spaces} implies that these also consist only of $\T$-positive Anosov representations: For the natural embedding $\SO(p,p+1)\to \SO(p,p+2)$ these components inject into $\T$-positive components of $\SO(p,p+2)$, in particular all representations are positive, when composed with the embedding $\SO(p,p+1)\to \SO(p,p+2)$. As a result those representations are also $\T$-positive in $\SO(p,p+1)$.}}
		It would  be interesting to construct explicit \emph{hybrid} representations in these components, in analogy to \cite{GWtop}. See \cite{Kyd} for results in this direction from the Higgs bundle perspective.
	\end{itemize}
\end{remark}

\appendix
\section{Positive representations in the split group $\PO(p,p)$}\label{a.POpp}
We restricted throughout the paper to the case $p<q$ because the root system of the split group $\PO(p,p)$ has a different expression, and, in this case, the ratio $\log\left|\frac{\lambda_{p-1}}{\lambda_p}\rho(g)\right|$ is not a simple root of $\PO(p,p)$, but rather the minimum of the last two simple roots. Furthermore the manifold $\calF_{p-1}$ we deal with in the paper identifies, in this case, with the full flag manifold (the stabilizer of a point is a Borel subgroup); this is due to the fact that in a space of signature $(p,p)$ every $(p-1)$-dimensional isotropic subspace is contained in precisely two $p$-dimensional isotropic subspaces. 

However it is easy to check that we never use the root system formalism in the paper, but only the linear algebraic interpretation of the group $\PO(p,q)$ and of the associated flags of isotropic subspaces. Moreover if $p=q$, then the $\T$-positive structure we describe is the positive structure of the split group $\PO(p,p)$ (see e.g. the proof of \cite[Proposition 7.10]{Collier-ENS}). As a result, our proofs of Theorems~\ref{thm.intro-positive},~\ref{thm.intro-collar} and~\ref{thm.C-intro} work in the case $p=q$ as well.\footnote{Alternatively one can deduce these results from the results of this paper composing a representation in $\PO(p,p)$ with the standard  embedding $\PO(p,p)\to\PO(p,p+1)$; the image under this homomorphism of the composition of an hyperbolization and the principal $\SL(2,\R)$-embedding is then $\Theta$ positive, and it thus follows from Corollary~\ref{cor.intro} that the whole $\PO(p,p)$-Hitchin component only consists of $\T$-positive Anosov representations.} 
In particular this shows that the Hitchin component in   $\PO(p,p)$, which so far had been elusive, only consists of Anosov representations (Theorem~\ref{thm.C-intro}) which are furthermore positively ratioed with respect to all symmetrized weights (Theorem~\ref{thm.intro-positive}) and satisfy root versus (symmetrized) weight collar lemmas for all simple roots in the root system: Theorem~\ref{thm.intro-collar} for $k=p-1$ compares the minimum of the last two roots to the symmetrized fundamental weight $\omega_{p-1}+\omega_p$ (i.e an opposition involution invariant weight). Taking the minimum of the roots can be dropped to get collar lemmas comparing the last two simple roots to the symmetrized weight. We believe that all these results are new for $\PO(p,p)$-Hitchin representations.


\newcommand{\etalchar}[1]{$^{#1}$}


\begin{thebibliography}{AABC{\etalchar{+}}19}
	
	\bibitem[AABC{\etalchar{+}}19]{ABC-IM}
	Marta Aparicio-Arroyo, Steven Bradlow, Brian Collier, Oscar Garc{\'i}a-Prada,
	Peter~B. Gothen, and Andr{\'e} Oliveira.
	\newblock So(p,q)-{H}iggs bundles and higher {T}eichm{\"u}ller components.
	\newblock {\em Inventiones mathematicae}, 218(1):197--299, Oct 2019.
	
	\bibitem[BCD{\etalchar{+}}08]{EinPrimer}
	Thierry Barbot, Virginie Charette, Todd Drumm, William~M. Goldman, and Karin
	Melnick.
	\newblock A primer on the {$(2+1)$} {E}instein universe.
	\newblock pages 179--229, 2008.
	
	\bibitem[BCG{\etalchar{+}}21]{BCGPGO}
	Steve {Bradlow}, Brian {Collier}, Oscar {Garcia-Prada}, Peter {Gothen}, and
	Andr{\'e} {Oliveira}.
	\newblock {A general Cayley correspondence and higher Teichm{\"u}ller spaces}.
	\newblock {\em arXiv e-prints}, page arXiv:2101.09377, January 2021.
	\newblock To appear Annals of Mathematics.
	
	\bibitem[BGPG06]{HiggsBundles}
	Steven~B. Bradlow, Oscar Garc\'{\i}a-Prada, and Peter~B. Gothen.
	\newblock Maximal surface group representations in isometry groups of classical
	{H}ermitian symmetric spaces.
	\newblock {\em Geom. Dedicata}, 122:185--213, 2006.
	
	\bibitem[BIPP21a]{BIPP0}
	M.~Burger, A.~Iozzi, A.~Parreau, and M.~B. Pozzetti.
	\newblock Currents, systoles, and compactifications of character varieties.
	\newblock {\em Proc. Lond. Math. Soc. (3)}, 123(6):565--596, 2021.
	
	\bibitem[BIPP21b]{BIPP}
	Marc Burger, Alessandra Iozzi, Anne Parreau, and Maria~Beatrice Pozzetti.
	\newblock The real spectrum compactification of character varieties:
	characterizations and applications.
	\newblock {\em Comptes Rendus. Math\'ematique}, 359(4):439--463, 2021.
	
	\bibitem[BIPP23]{BIPPrsp}
	Marc {Burger}, Alessandra {Iozzi}, Anne {Parreau}, and Maria~Beatrice
	{Pozzetti}.
	\newblock {The real spectrum compactification of character varieties}.
	\newblock {\em arXiv e-prints}, page arXiv:2311.01892, November 2023.
	
	\bibitem[BIW10]{BIW}
	Marc Burger, Alessandra Iozzi, and Anna Wienhard.
	\newblock Surface group representations with maximal {T}oledo invariant.
	\newblock {\em Ann. of Math. (2)}, 172(1):517--566, 2010.
	
	\bibitem[Bou02]{Bou2}
	Nicolas Bourbaki.
	\newblock {\em Lie groups and {L}ie algebras. {C}hapters 4--6}.
	\newblock Elements of Mathematics (Berlin). Springer-Verlag, Berlin, 2002.
	\newblock Translated from the 1968 French original by Andrew Pressley.
	
	\bibitem[BP17]{BP}
	Marc Burger and Maria~Beatrice Pozzetti.
	\newblock Maximal representations, non-{A}rchimedean {S}iegel spaces, and
	buildings.
	\newblock {\em Geom. Topol.}, 21(6):3539--3599, 2017.
	
	\bibitem[BP21]{Beyrer-Pozzetti}
	Jonas Beyrer and Beatrice Pozzetti.
	\newblock A collar lemma for partially hyperconvex surface group
	representations.
	\newblock {\em Trans. Amer. Math. Soc.}, 374(10):6927--6961, 2021.
	
	\bibitem[BP24]{Beyrer-Pozzetti2}
	Jonas Beyrer and Beatrice Pozzetti.
	\newblock Degenerations of -positive surface group representations.
	\newblock {\em Journal of Topology}, 17(3):e12352, 2024.
	
	\bibitem[Col20]{Collier-ENS}
	Brian Collier.
	\newblock {$ {SO}(n, n+1)$}-surface group representations and {H}iggs bundles.
	\newblock {\em Ann. Sci. \'{E}c. Norm. Sup\'{e}r. (4)}, 53(6):1561--1616, 2020.
	
	\bibitem[CTT19]{CTT}
	Brian Collier, Nicolas Tholozan, and J\'{e}r\'{e}my Toulisse.
	\newblock The geometry of maximal representations of surface groups into {${\rm
			SO}_0(2,n)$}.
	\newblock {\em Duke Math. J.}, 168(15):2873--2949, 2019.
	
	\bibitem[DS20]{DS}
	David Dumas and Andrew Sanders.
	\newblock Geometry of compact complex manifolds associated to generalized
	quasi-{F}uchsian representations.
	\newblock {\em Geom. Topol.}, 24(4):1615--1693, 2020.
	
	\bibitem[FG06]{FG}
	Vladimir Fock and Alexander Goncharov.
	\newblock Moduli spaces of local systems and higher {T}eichm\"{u}ller theory.
	\newblock {\em Publ. Math. Inst. Hautes \'{E}tudes Sci.}, (103):1--211, 2006.
	
	\bibitem[FP20]{FP}
	Federica Fanoni and Maria~Beatrice Pozzetti.
	\newblock Basmajian-type inequalities for maximal representations.
	\newblock {\em J. Differential Geom.}, 116(3):405--458, 2020.
	
	\bibitem[GGKW17]{GGKW}
	Fran\c{c}ois Gu\'{e}ritaud, Olivier Guichard, Fanny Kassel, and Anna Wienhard.
	\newblock Anosov representations and proper actions.
	\newblock {\em Geom. Topol.}, 21(1):485--584, 2017.
	
	\bibitem[GLW21]{GLW}
	Olivier {Guichard}, Fran{\c{c}}ois {Labourie}, and Anna {Wienhard}.
	\newblock {Positivity and representations of surface groups}.
	\newblock {\em arXiv e-prints}, page arXiv:2106.14584, June 2021.
	
	\bibitem[GW10]{GWtop}
	Olivier Guichard and Anna Wienhard.
	\newblock Topological invariants of {A}nosov representations.
	\newblock {\em J. Topol.}, 3(3):578--642, 2010.
	
	\bibitem[GW12]{Guichard-Wienhard-IM}
	Olivier {Guichard} and Anna {Wienhard}.
	\newblock {Anosov representations: domains of discontinuity and applications.}
	\newblock {\em {Invent. Math.}}, 190(2):357--438, 2012.
	
	\bibitem[GW18]{GWpositivity}
	Olivier Guichard and Anna Wienhard.
	\newblock Positivity and higher {T}eichm\"{u}ller theory.
	\newblock In {\em European {C}ongress of {M}athematics}, pages 289--310. Eur.
	Math. Soc., Z\"{u}rich, 2018.
	
	\bibitem[GW22]{GWPaper}
	Olivier {Guichard} and Anna {Wienhard}.
	\newblock {Generalizing Lusztig's total positivity}.
	\newblock {\em arXiv e-prints}, page arXiv:2208.10114, August 2022.
	
	\bibitem[Hit92]{Hitchin}
	N.~J. Hitchin.
	\newblock Lie groups and {T}eichm\"{u}ller space.
	\newblock {\em Topology}, 31(3):449--473, 1992.
	
	\bibitem[HS23]{HuangSun}
	Yi~Huang and Zhe Sun.
	\newblock Mc{S}hane identities for higher {T}eichm\"uller theory and the
	{G}oncharov-{S}hen potential.
	\newblock {\em Mem. Amer. Math. Soc.}, 286(1422):v+116, 2023.
	
	\bibitem[Kee74]{Keen}
	Linda Keen.
	\newblock Collars on {R}iemann surfaces.
	\newblock In {\em Discontinuous groups and {R}iemann surfaces ({P}roc. {C}onf.,
		{U}niv. {M}aryland, {C}ollege {P}ark, {M}d., 1973)}, pages 263--268. Ann. of
	Math. Studies, No. 79, 1974.
	
	\bibitem[KLP17]{KLP}
	Michael Kapovich, Bernhard Leeb, and Joan Porti.
	\newblock Anosov subgroups: dynamical and geometric characterizations.
	\newblock {\em European Journal of Mathematics}, 3(4):808--898, Dec 2017.
	
	\bibitem[KP20]{PK}
	Fanny Kassel and Rafael Potrie.
	\newblock Eigenvalue gaps for hyperbolic groups and semigroups.
	\newblock {\em arXiv e-prints, arXiv:2002.07015}, 2020.
	
	\bibitem[Kyd22]{Kyd}
	Georgios Kydonakis.
	\newblock From hyperbolic {D}ehn filling to surgeries in representation
	varieties.
	\newblock pages 201--260, [2022] \copyright 2022.
	
	\bibitem[Lab06]{Labourie-IM}
	Fran\c{c}ois Labourie.
	\newblock Anosov flows, surface groups and curves in projective space.
	\newblock {\em Invent. Math.}, 165(1):51--114, 2006.
	
	\bibitem[LM09]{Lab-McShane}
	Fran\c{c}ois Labourie and Gregory McShane.
	\newblock Cross ratios and identities for higher {T}eichm\"{u}ller-{T}hurston
	theory.
	\newblock {\em Duke Math. J.}, 149(2):279--345, 2009.
	
	\bibitem[Lus94]{Lus}
	G.~Lusztig.
	\newblock Total positivity in reductive groups.
	\newblock In {\em Lie theory and geometry}, volume 123 of {\em Progr. Math.},
	pages 531--568. Birkh\"{a}user Boston, Boston, MA, 1994.
	
	\bibitem[LZ17]{Lee-Zhang}
	Gye-Seon Lee and Tengren Zhang.
	\newblock Collar lemma for {H}itchin representations.
	\newblock {\em Geom. Topol.}, 21(4):2243--2280, 2017.
	
	\bibitem[MZ19]{MZ}
	Giuseppe {Martone} and Tengren {Zhang}.
	\newblock {Positively ratioed representations.}
	\newblock {\em {Comment. Math. Helv.}}, 94(2):273--345, 2019.
	
	\bibitem[Par12]{APcomp}
	A.~Parreau.
	\newblock Compactification d'espaces de repr\'esentations de groupes de type
	fini.
	\newblock {\em Math. Z.}, 272(1-2):51--86, 2012.
	
	\bibitem[PS17]{PS}
	Rafael Potrie and Andr\'{e}s Sambarino.
	\newblock Eigenvalues and entropy of a {H}itchin representation.
	\newblock {\em Invent. Math.}, 209(3):885--925, 2017.
	
	\bibitem[PSW21]{PSW1}
	Maria~Beatrice Pozzetti, Andr\'{e}s Sambarino, and Anna Wienhard.
	\newblock Conformality for a robust class of non-conformal attractors.
	\newblock {\em J. Reine Angew. Math.}, 774:1--51, 2021.
	
	\bibitem[PSW23]{PSW2}
	Maria~Beatrice Pozzetti, Andr\'es Sambarino, and Anna Wienhard.
	\newblock Anosov representations with {L}ipschitz limit set.
	\newblock {\em Geom. Topol.}, 27(8):3303--3360, 2023.
	
	\bibitem[RS90]{RS}
	R.~W. Richardson and P.~J. Slodowy.
	\newblock Minimum vectors for real reductive algebraic groups.
	\newblock {\em J. London Math. Soc. (2)}, 42(3):409--429, 1990.
	
	\bibitem[VY17]{Vlamis-Yarmola}
	Nicholas~G. Vlamis and Andrew Yarmola.
	\newblock Basmajian's identity in higher {T}eichm\"{u}ller-{T}hurston theory.
	\newblock {\em J. Topol.}, 10(3):744--764, 2017.
	
	\bibitem[ZZ24]{ZZ}
	Tengren Zhang and Andrew Zimmer.
	\newblock Regularity of limit sets of anosov representations.
	\newblock {\em Journal of Topology}, 17(3):e12355, 2024.
	
\end{thebibliography}
\end{document}